\documentclass[12pt]{article}
\usepackage{setspace} 
\usepackage[margin=1.05in]{geometry}
\usepackage{graphicx}
\usepackage{caption}
\usepackage{mathtools}
\usepackage{amsmath}
\usepackage{amssymb}

\usepackage{lineno}
\usepackage{transparent}
\usepackage{amsthm}
\usepackage{amscd}
\usepackage{hyperref}
\usepackage{hyperref}
\theoremstyle{plain}
\newtheorem{theorem}{Theorem}[section]
\newtheorem{proposition}[theorem]{Proposition}
\newtheorem{lemma}[theorem]{Lemma}
\newtheorem{corollary}[theorem]{Corollary}
\theoremstyle{definition}
\newtheorem{defn}[theorem]{Definition}
\newtheorem{example}[theorem]{Example}
\numberwithin{figure}{section}
\newcommand{\R}{\mathbb{R}}

\newcommand{\C}{\mathbb{C}}

\usepackage{tikz}
\usepackage{lineno}
\usepackage{tikz-cd}
\usepackage{amsmath}
\usepackage{amssymb}
\usepackage{amsthm}
\usepackage{amscd}
\usepackage{hyperref}
\usepackage[nottoc]{tocbibind}
\usepackage{amsthm}
\usepackage{hyperref}
\theoremstyle{plain}
\usepackage{amsthm}
\usepackage{hyperref}
 \usepackage{setspace}

\newtheorem{remark}[theorem]{Remark}

\usepackage[utf8]{inputenc}
\usepackage{amssymb,amscd}
\usepackage{pstricks,pstricks-add}
\usepackage{tqft}
\usepackage{tikz}
\usepackage{hyperref}
\usepackage{amsmath}
\usepackage{amssymb}
\usepackage{blindtext}
\usepackage{geometry}
\usepackage[T1]{fontenc}
\usepackage{lmodern}
\usepackage{theoremref}
\usepackage[T1]{fontenc}

\usepackage{graphicx}
\graphicspath{ {./images/} }
\usepackage{csquotes}
\usepackage{graphicx} 
\usepackage[rgb]{xcolor} 
\usepackage{tikz}
\usepackage{eucal}
\usepackage{enumerate}
\usepackage{subfigure}
\usepackage{cite}
\usepackage[font=sc,justification=centering]{caption}
\usetikzlibrary{patterns}
\makeatletter
\tikzset{
        hatch distance/.store in=\hatchdistance,
        hatch distance=5pt,
        hatch thickness/.store in=\hatchthickness,
        hatch thickness=5pt,
        every point/.style = {circle, inner sep={.75\pgflinewidth}, opacity=1, draw, solid, fill=white},
        point/.style={insert path={node[every point, #1]{}}}, point/.default={}
        }
\pgfdeclarepatternformonly[\hatchdistance,\hatchthickness]{north east hatch}
    {\pgfqpoint{-1pt}{-1pt}}
    {\pgfqpoint{\hatchdistance}{\hatchdistance}}
    {\pgfpoint{\hatchdistance-1pt}{\hatchdistance-1pt}}%
    {
        \pgfsetcolor{\tikz@pattern@color}
        \pgfsetlinewidth{\hatchthickness}
        \pgfpathmoveto{\pgfqpoint{0pt}{0pt}}
        \pgfpathlineto{\pgfqpoint{\hatchdistance}{\hatchdistance}}
        \pgfusepath{stroke}
    }
\makeatother
\usepackage{import}
\usepackage{setspace}

\DeclareMathOperator{\Ker}{\mathrm{Ker}}

\DeclareMathOperator{\Sym}{Sym}

\DeclareMathOperator{\skel}{skel}
\DeclareMathOperator{\PP}{\mathbb{P}}

\DeclareMathOperator{\bR}{\mathbb{R}}

\newcommand{\CC}{\mathbb C}

\newtheorem*{theoremA}{Theorem A}
\newtheorem*{theoremB}{Theorem B}

\usepackage{mathpazo}
\title{Sectorial Decompositions of Symmetric
Products of Surfaces}
\author{Xinle Dai}
\date{}
\begin{document}
\maketitle
\onehalfspacing

\begin{abstract}
Symmetric products of Riemann surfaces play a crucial role in symplectic geometry and low-dimensional topology. Symmetric products of punctured surfaces are Liouville manifolds of interest e.g. for Heegaard Floer theory. We study the symplectic topology of these spaces using Liouville sectorial techniques, along with examples and applications of these decompositions in the context of homological mirror symmetry. More precisely, we show that a sectorial decomposition of a Riemann surface along a union of arcs induces a sectorial decomposition of its second symmetric product and as an application, we give a new geometric proof of Homological Mirror Symmetry for the complex two dimensional pair of pants.
\end{abstract}
\tableofcontents
\section{Introduction}

Mirror symmetry was first observed as an unexpected phenomenon in string theory in the late $1980$s. This phenomenon bridges two seemingly unrelated areas of mathematics: algebraic geometry and symplectic geometry. The resulting conjectures have since revolutionized both fields. In $1994$, at the ICM, Kontsevich \cite{Kon94} proposed that mirror symmetry reflects an equivalence of two categories
\[
D^\pi \mathrm{Fuk} (X) \cong D^b \mathrm{Coh}(\check{X}),
\]
between the derived Fukaya category of a Calabi--Yau manifold $X$  and the derived category of coherent sheaves on the mirror Calabi-Yau $\check{X}$. This formulation, known as homological mirror symmetry, has been generalized beyond Calabi-Yau manifolds. While verifying homological mirror symmetry for various examples is an active field, it also serves as a tool for transferring ideas and insights between symplectic and algebraic geometry.

A particularly timely direction in the study of homological mirror symmetry is to understand how Fukaya categories and the equivalence conjectured by Kontsevich behave under geometric decompositions of spaces into simpler building blocks. The setting in which this has been most successful so far are Liouville manifolds, which are non-compact symplectic manifolds with exact symplectic structure and cylindrical ends. 
However, when Liouville manifolds have boundaries, defining Floer homology becomes challenging because the pseudoholomorphic curves (whose counts are used to define operations in Floer theory) might escape through the boundary. This leads to the notion of \emph{Liouville sector}, first introduced in a recent work of Ganatra-Pardon-Shende \cite{MR4106794}.

This paper explores the application of this technology to the study of symmetric products of open Riemann surfaces, which are smooth complex manifolds. It is a particularly important class of Liouville manifolds in light of applications to low-dimensional topology, such as Heegaard Floer homology by Ozsv\'ath and Szab\'o \cite{MR2113020}. Tim Perutz \cite{Per05} gave a careful description of a symplectic structure on the symmetric product. (The naive symplectic structure on the symmetric product induced by the symplectic form on the surface is not smooth on the diagonal.) Auroux \cite{MR2755992, MR2827825} provided a symplectic interpretation of bordered Heegaard-Floer homology through Fukaya categories of symmetric products. This application provides further motivation to study the symplectic geometry and the Fukaya categories of symmetric products.

We study the geometry of the second symmetric products of surfaces by constructing a decomposition of these spaces into Liouville sectors. This provides a symplectic insight into the structure of symmetric products and further exploration of its applications of the machinery of sectors in homological mirror symmetry. The main result of this paper is the following geometric decomposition theorem, which is restated and proved in section $5$.

\begin{theoremA}
Let $\Sigma$ be a topological surface with disjoint proper embedded arcs $\{\gamma_i\}_{i\in I}$. For suitably defined Stein structures on $\Sigma$ and $\Sym^2(\Sigma)$, the arcs $\{\gamma_i\}_{i\in I}$ determines a sectorial collection of smooth hypersurfaces $H_{_i}$ in $\Sym^2(\Sigma)$ which separates the symmetric product into Liouville sectors. 
\end{theoremA}

Roughly, $H_i$ corresponds to the set of configurations of points on $\Sigma$ where one of the points is on the arc $\gamma_i$. However, it is difficult to construct a well-behaved geometric decomposition of the symmetric product directly from a geometric decomposition of $\Sigma$, because symmetric products of manifolds with boundaries have badly behaved corners and singularities.

In general, this theorem states that given a surface's sectorial decomposition, we can construct a sectorial decomposition of its second symmetric product.

The proof of theorem A relies on two main ingredients. One is the observation that the ascending manifold of an index one critical point of a Stein potential defines a sectorial hypersurface. The other ingredient is the
calculation on a local model described in section $3$. The methods used to prove Theorem A also imply other results for sectors and symmetric products of Riemann surfaces. For example,

\begin{theoremB}
For any $2$-dimensional Liouville sector $X$, the second symmetric product of $X$ is deformation equivalent to a Liouville sector $Y$. Moreover, the completion $\widehat{Y}$ is deformation equivalent to the second symmetric product $Sym^2(\widehat X)$ of the completion of $X$.
\end{theoremB}
Furthermore, we can apply the decomposition theorem (Theorem A) to homological mirror symmetry. Namely, a sectorial decomposition of $\Sym^2(\Sigma)$ opens the perspective of constructing a mirror space of the symmetric product by finding mirrors to the various Liouville sectors into which it can be decomposed and gluing them to each other in a suitable manner.

Ganatra, Pardon, and Shende \cite{MR4106794, MR4695507} have introduced and studied  wrapped Fukaya categories of Liouville sectors. The upshot of their work is that the wrapped Fukaya category of a Liouville manifold can be determined from a decomposition into Liouville sectors by a local-to-global principle.

In section $7$, we discuss the example of the symmetric product of the $4$-punctured sphere $\Sym^2(\mathbb{P}^1 - 4 \mbox{pts})$, which is also known as the \emph{complex $2$-dimensional pair of pants}. Lekili and Polishchuk \cite{MR4120165} have established homological mirror symmetry for this example using a computation of Auslander algebras. A geometric interpretation of mirror symmetry can be deduced from our sectorial decomposition results. In this case, the local-to-global principle in  \cite{MR4695507} yields the top diamond of the commutative diagram shown in Figure \ref{marker},
which describes the wrapped Fukaya category of $\text{Sym}^2(\mathbb{P}^1 - 4\text{pts})$ as the pushout of a diagram of wrapped Fukaya
categories of simpler Liouville sectors. Homological Mirror Symmetry for the $2$-dimensional pair of pants follows by identifying the mirror of each Liouville sector and the gluing maps between them.

The rest of this paper is organized as follows.
Section $2$ reviews the basic theory of Liouville sectors and related notions such as sectorial domains, completions and truncations.
Section $3$ develops a local model on the symmetric square of the complex plane while section $4$ establishes the existence of quadratic Stein structures on surfaces with embedded arcs, providing the analytic framework for our main construction.
Section $5$ uses this structure to construct a global sectorial decomposition of the symmetric product of a surface.
Section $6$ studies the fibers of the sectorial boundaries and their completions, showing their Liouville equivalences.
Finally, Section $7$ applies these results to the symmetric product of the four-punctured sphere and proves homological mirror symmetry for this example.

\paragraph{Acknowledgement.}
I would like to thank my PhD advisor Denis Auroux for his guidance, support, and generosity, as well as for the many ideas he contributed to this work. I am also grateful to John Pardon for helpful conversations and suggestions. This work was partially supported by NSF grant DMS-2202984, by the Simons Foundation (grant \#$385573$, Simons Collaboration on Homological Mirror Symmetry), and by the Simons Collaboration on New Structures in Low-Dimensional Topology.

\section{Sector theory and more}

\subsection{Liouville Sectors} 
This section is mostly expository following Ganatra-Pardon-Shende \cite{MR4106794, MR4695507} and some classical concepts and results.

A \emph{symplectic manifold} is a pair $(X,\omega)$, where $X$ is an even dimensional real manifold and $\omega\in\Omega^2 (X)$ is a closed nondegenerate $2$-form on $X$. 

\begin{defn}
     A compact symplectic manifold $(X,\omega)$ with boundary $\partial X$ is a Liouville domain if $\omega$ is exact and there exists a vector field $Z$ on $X$ such that 
    \begin{enumerate}[(a)]
        \item $\mathcal{L}_Z\omega =\omega$, or equivalently, $Z$ is dual to a primitive form of $\omega$;
        \item  $Z$ is outward pointing along the boundary $\partial X$.  
    \end{enumerate}
	We call $Z$ a \emph{Liouville vector field} and $(\omega ,Z)$ a \emph{Liouville structure} on $X$. 
\end{defn}

\begin{defn}
    A symplectic manifold $(X,\omega)$ is a  \emph{Liouville manifold} if $\omega$ is exact and there exists a complete Liouville vector field $Z$ on $X$ such that it is cylindrical and convex at infinity.
\end{defn}

\begin{defn}
    Let $(X_0,\omega)$ be a Liouville domain, one can always construct a Liouville manifold $X$ given by $$
    X =X_0\cup_{\partial X_0}\partial X_0\times [1,+\infty),
$$
where the Liouville structure on $\partial X_0\times [1,+\infty)_r$ is $\left(\omega =d(r\lambda\vert_{\partial X_0}),Z=r\partial_r\right)$.
We call $X$ the completion of Liouville domain $(X_0,\omega)$ 
\end{defn}

\begin{defn}
Let $(X,\omega, Z)$ be a Liouville manifold. If there exists an exhausting Morse function $\varphi\colon X\to\mathbb{R}$ that is Lyapunov (i.e. $Z$ is gradient-like for $\varphi$) for $Z$, we then call $X$ a  Weinstein manifold. The triple  $(\omega ,Z,\varphi)$ is called a \emph{Weinstein structure} on $X$.
\end{defn} 

\begin{defn}
 A \emph{Stein manifold} is a complex manifold $(X,J)$ which admits a proper holomorphic embedding $i\colon X\to\mathbb{C}^N$ for some $N$.
\end{defn}


\begin{defn}[Ganatra-Pardon-Shende]
Let $ X $ be a Liouville manifold with boundaries, i.e., $ (X, \lambda) $ is an exact symplectic manifold with boundaries, with infinite ends modeled on the symplectization of a contact manifold with boundary.

We say $ X $ is a \textit{Liouville sector} iff there exists a function $I : \partial X \rightarrow \mathbb{R}$ such that it satisfies the following conditions.

\begin{itemize}
  \item $I$ is \textit{linear at infinity}, meaning $ZI = \alpha I$ outside a compact set for some constant $\alpha >0$, where $Z$ denotes the Liouville vector field.
  \item $dI|_{\text{char.fol.}} > 0$, where the characteristic foliation $C$ of $\partial X$ is oriented so that $\omega(N, C) > 0$ for any inward pointing vector $N$.
\end{itemize}
\end{defn}

The boundary of a Liouville sector is identified with $F\times \bR$, where $F:=I^{-1}(0)$ is called the fiber.

On any Liouville sector $X$, there is a canonical (up to contractible choice) symplectic fibration $\pi : X \rightarrow \mathbb{C}_{\text{Re}\geq 0}$ defined near $\partial X$. For almost complex structures on $X$ making $\pi$ holomorphic, the projection $\pi$ imposes strong control on holomorphic curves near $\partial X$.

See \cite{MR4106794} section~$2.6$ for a description of the local model near the boundary~$\partial X$ of~$X$, and for the discussion of the family $Z_\CC^\alpha = (1-\alpha)x\partial_x + \alpha y\partial_y$ of Liouville vector fields on~$\CC$ and its associated Liouville forms~$\lambda_\CC^\alpha$. In particular, the case $\alpha = \tfrac12$ gives the radial Liouville structure, under which $J_\CC$ is invariant.

\begin{defn}
    A saddle sector is a Liouville sector with $\alpha>1$.
\end{defn}
The name of saddle sector comes from the dynamics of the Liouville flow in the $(R,I)$ coordinates, which has a saddle point at the origin where $R=I=0$.

Here are some examples of Liouville sectors.

\begin{example}
\begin{enumerate}
\item The cotangent bundle $T^*Q$ for any compact manifold-with-boundary $Q$ is a Liouville sector.  
\item Any Liouville manifold is an example of Liouville sector.
\item A punctured bordered Riemann surface $S$ is a Liouville sector if and only if every component of $\partial S$ is homeomorphic to $\mathbb{R}$, that is none of the boundary components is homeomorphic to $S^1$.
\item
If $ X $ is a Liouville sector, $ \Lambda \subset \partial_{\infty} X $ closed Legendrian, then $ X - \mathcal{N}_\epsilon (\Lambda) $ is a Liouville sector.
\item In section $4$ definition \ref{D_1minus}, we will define and study an explicit Liouville sector structure on $D_1^-:=\{ z \in \mathbb{C}| \text{Re}(z) \leq 1\}$.
\end{enumerate}
  
\end{example}

\begin{defn}[Sectorial covering]\label{sector_covers}
Let $X$ be a Liouville manifold-with-boundary. 
A collection of cylindrical hypersurfaces $H_1,\cdots,H_n\subseteq X$ is called \emph{sectorial}
iff their characteristic foliations are $\omega$-orthogonal over their intersections and there exist functions
$I_i:Nbd^ZH_i\to\mathbb{R}$ which is linear near infinity
satisfying
\begin{enumerate}
    \item $dI_i|_{char.fol.(H_i)}\ne 0$ (equivalently, $X_{I_i}\notin TH_i$), 
    \item  $dI_i|_{char.fol.(H_j)}=0$ (equivalently, $X_{I_i}\in TH_j$), 
    \item  and $\{I_i,I_j\}=0$.
\end{enumerate}
Let $X$ be a Liouville manifold or sector and $X_1,\ldots,X_n$ in $X$ be a covering of a $X$ by finitely many Liouville sectors. This covering is called sectorial iff the collection of their boundaries $\partial X,\partial X_1,\cdots,\partial X_n$ is sectorial. Here we understand $\partial X_i$ as the part of the boundary of $X_i$ not coming from $\partial X$, i.e. the boundary in the point-set topological sense.
\end{defn}

\begin{remark} If we have more than one hypersurface which intersect each other, we can define sectors with corners and the local model as follows. Let $(M, \lambda)$ be a Liouville manifold, and suppose that we have a collection of Liouville hypersurfaces 
$
H_1, H_2, \dots, H_k \subset M
$
which intersect transversely. For each $i$, we write
$
H_i \cong F_i \times \mathbb{R},
$
where $F_i$ denotes the Liouville fiber of $H_i$. When two hypersurfaces intersect, say $H_i$ and $H_j$, we denote their intersection by
$
C_{i,j} := H_i \cap H_j \cong (F_i \cap F_j) \times \mathbb{R}^2.
$. 

In dimension $4$, there is a simple local model for the neighborhood of any corner using coordinates $z_1=R_1+iI_1$ and $z_2=R_2+iI_2$.
The sector with corners is $\{ R_1 \leq 0\} \times \{ R_2 \leq 0\}$. The fibers are $F_1 = \{ I_1=0\} \cap H_1 \cong \{ z_1= 0\} \times \{ R_2 \leq 0\} $ and $F_2  \cong \{ I_2=0\} \cap H_2\cong \{ R_1 \leq 0\} \times  \{z_2=  0\}$
and they intersect at $I_1=I_2=0$.

\end{remark}


\subsection{Completions}

In this section, we want to introduce three kinds of completions and show that two of them are equivalent. These constructions are three different ways (all equivalent up to deformation) of turning a saddle sector into a convex Liouville manifold. 

Let $X$ be a saddle sector with the saddle boundary $S$, we will define three kinds of completions, i.e. Liouville completion, convex completion, and direct convex completion. We call these three completions $X_1$, $X_2$, and $X_3$.

\begin{enumerate}
    \item \textbf{Liouville Completion: gluing a Minimal Sector to get $X_1$}: \\We glue a minimal sector, specifically, the product of the fiber with a model saddle sector $D_1^- = \{ z \in \mathbb{C}| Re(z) \leq 1\}$ (see Definition \ref{D_1minus}), to $X$ along the boundary of the saddle point. This construction fills in the boundary near the saddle and yields a completed manifold denoted by $X_1$. This process can be viewed as a natural extension of $X$ using a minimal Liouville sector.

    \item \textbf{Convex Completion: deleting a tubular neighborhood and performing convex completion to get $X_2$}: \\We first remove a tubular neighborhood $N(S)$ around the boundary $S$ of $X$, where $S$ denotes the sector boundary associated with the saddle point. The remaining manifold $X_{-}:=X \setminus N(S)$ is then completed to make the boundary convex. This is constructed as $X_{-} \cup ([0, \infty) \times \partial X_{-})$, where $\partial X_{-}$ is a convex type hypersurface. The completed manifold is denoted by $X_2$. This method ensures that the boundary satisfies the convexity conditions required for a Liouville domain.

    \item \textbf{Direct Convex Completion to get $X_3$}: By directly removing the sector boundary $S$ of $X$, we obtain a complete Liouville manifold. This approach avoids the intermediate step of removing a tubular neighborhood and results in a manifold which is complete and convex at infinity. The completed manifold is denoted by $X_3$.
\end{enumerate}


The relationships among these completions are as follows.
\begin{proposition}
Let $X_1$, $X_2$, and $X_3$ denote the three completions of a saddle sector $X$. The relationships among these completions are as follows.
\begin{enumerate}
    \item The completion $X_1$ is Liouville deformation equivalent to $X_2$. That is, there exists a smooth deformation of the Liouville structure of $X_1$ that yields $X_2$.
    \item The completion $X_3$ is Liouville deformation equivalent to both $X_1$ and $X_2$. 
\end{enumerate}
\end{proposition}
\begin{remark}
Although $X_3$ is Liouville equivalent to $X_2$, it is not Stein isomorphic to $X_2$ as $X_3$ doesn't have a natural Stein structure. This distinction arises because modifying the complex structure potential is necessary to equip $X_3$ with a Stein structure.    
\end{remark}

\subsection{Sector domains and truncations}

We aim to define an operation called \textit{truncation}, which involves creating a subdomain with a new convex boundary such that the negative Liouville flow carries every point into this subdomain.

\begin{defn}[Sector Domain]\label{sectordomain}
We define a sector domain $D$ to be a Liouville manifold (possibly non-compact) with boundaries (with possibly corners) such that one part of the boundary $\partial_C(D)$ is convex (contact-type) and the other part of the boundary $\partial_S(D)$ is a sectorial boundary.
\end{defn}

\begin{remark}
    We can think of a sector domain as shrinking a Liouville sector along the negative Liouville vector field, and the converse operation is Liouville completion along the contact-type boundary of the sector domain. 
\end{remark}

\begin{defn}
For a Liouville sector \( X \), we say that a sector domain \( D \subset X \) is a \emph{truncation} if for each \( x \in X \), there exists \( t_x  >0 \) such that
\[
\varphi^{-t_x}(x) \in D,
\]
and at every point of $\partial_S  X - \partial_S D$, the Liouville vector field $Z$ is tangent to $\partial_S X$ and $ZI=\alpha I$.

\end{defn}

\begin{defn}
The \emph{completion} of a sectorial domain \( D \) is a Liouville sector
\[
\widehat{D} := \mathsf{C}(D) := D \cup_{\partial_C D} (\partial_C D \times [1, \infty)).
\]
\end{defn}

\begin{proposition}
The completion \( \widehat{D} \) of a truncation \( D \subset X \) is symplectomorphic to \( X \) as a Liouville sector
\[
\mathsf{C}(D) \cong X.
\]
\end{proposition}

\section{The local model} 
Let $\varphi$ be a plurisubharmonic Morse function on $\mathbb{C}$. Let $\omega = dd^c \varphi$ and the Liouville form be $d^c \varphi$. There exists a Riemannian metric $g_\varphi$ determined by $\omega$ and the almost complex structure. We then get a Liouville vector field $Z= \nabla {g_\varphi} \varphi$.

Let us use coordinate $z=x+iy$ and assume $$\varphi (z) = \frac{1}{4} \| z\|^2 + \frac{1-2\alpha}{4} \mbox{ Re}(z^2)=\frac{1-\alpha}{2} x^2+ \frac{\alpha}{2}y^2,$$ where $\alpha >1$. This Morse function has only one index $1$ critical point at $0$. 

For the computation simplicity, let us choose $\alpha=\frac{3}{2}$, we then have $$\varphi(z)= -\frac{1}{4}x^2 + \frac{3}{4}y^2.$$

 Sym$^2(\mathbb{C}) \cong \mathbb{C}^2$, we can change the coordinate to $z=\frac{1}{2}(z_1+z_2)$ and $w= (\frac{z_1-z_2}{2})^2.$

Then by the Taylor expansion $\varphi(z_1)+\varphi(z_2) = 2 \varphi(z) + 2\varphi(\sqrt{w})$.

We define a  function $\Phi$ on Sym$^2(\mathbb{C})$  $$\Phi(z, w) := 2 \varphi(z) + 2\varphi(\sqrt{w})= \frac{1}{2} |z|^2- \mbox{Re}(z^2) + \frac{1}{2} |w| -  \mbox{Re}(w).$$
We notice that the function $\Phi$ we defined above has a cone singularity at $0$ for $|w|$, we regularize the singularity by a compactly supported perturbation to get a smooth function $\widetilde{\Phi}$ on Sym$^2(\mathbb{C})$.
$$\widetilde {\Phi}(z, w) := 2 \varphi(z) + 2\widetilde{\varphi}(\sqrt{w}):= \frac{1}{2} |z|^2- \mbox{Re}(z^2) + \frac{1}{2} \widetilde{|w|} -  \mbox{Re}(w).$$

The function $\widetilde {\Phi}(z, w)$ doesn't have any critical points.

For each of the coordinates $z, z_1, z_2$, and $\sqrt{w}$, we have the gradient flow of $\Phi$ as follows.
\begin{equation}
    \Phi_t(x,y)= (e^\frac{t}{2}x, e^{-\frac{3t}{2}}y) \label{exponential flow}
\end{equation}
Re($z$) goes to $-\infty$ or $ \infty$ unless Re$(z)=0$, and $z \to 0$ when Re$(z)=0$. For the regularized function $\widetilde{\Phi}$, the answer is similar except in the neighborhood of the diagonal $w=0$.
The gradient of $\widetilde{\Phi}$ is a decoupled flow of $z_1$ and $z_2$ when it is away from the diagonal (or $z$ and $\sqrt{w}$ always), this changes only near $w=0$ and where we smooth $|w|$. 

\begin{lemma}
For a suitable choice of smoothing, under the gradient flow of $\widetilde{\Phi}$, Re$(w)$ goes to $+\infty$ and Im$(w)$ goes to $0$, as $t$ goes to $\infty$. \label{escaping lemma}
\end{lemma}

\begin{proof}
    If we set the smoothing for $w=x+iy$ to be $\sqrt{x^2+y^2+\epsilon}$ for some small $\epsilon>0$, then the negative gradient of $\frac{1}{2}\widetilde{|w|}-\mbox{Re}(w)$ for the standard metric are $\dot{x}=\frac{x}{\sqrt{x^2+y^2+\epsilon}}+2$ and $\dot{y}=-\frac{y}{\sqrt{x^2+y^2+\epsilon}}$. We notice that $\dot{x}$ is always larger than $1$ and $\dot{y}$ is always has the same sign as negative $y$.
    
Now we compute the negative gradient of $\frac{1}{2}\widetilde{|w|}-\mbox{Re}(w)$ for the K\"ahler metric defined by $dd^c (\widetilde{\Phi})$. 
$$dd^c(\widetilde{\Phi})= dd^c(\frac{1}{2}|z|^2)+dd^c(\frac{1}{2}\widetilde{|w|})= idzd \Bar{z} + \frac{i}{4}\frac{|w|^2+2\epsilon}{(|w|^2+\epsilon)^{3/2}}dwd\Bar{w}$$ 

    The gradient for this K\"ahler metric is $\frac{2(x^2+y^2+\epsilon)^\frac{3}{2}}{x^2+y^2+2\epsilon}$ times the gradient for the standard metric.
    Since $\frac{2(x^2+y^2+\epsilon)^\frac{3}{2}}{x^2+y^2+2\epsilon}$ is always larger than $\sqrt{\epsilon}$, $\Dot{x}$ is bounded below by $\sqrt{\epsilon}>0$, then $x$ goes to $+\infty$. And similarly, $y$ goes to $0$.

\end{proof}

\begin{remark}
The proof would work more generally if the smoothing satisfies the following two properties.

\begin{itemize}
    \item $\widetilde{|w|}$ is strictly subharmonic.
    \item $\frac{\partial}{\partial x}(\frac{1}{2}\widetilde{|w|}-\mbox{Re}(w))< c<0$ everywhere for some constant $c$.
\end{itemize}

\end{remark}

We then can conclude that the difference of $z_1$ and $z_2$ will goes to $+\infty$. The gradient flow faster through the origin and doesn't stop at the origin.

For sufficient large $t$, the gradient flow line will be away from the diagonal and is the same in the \eqref{exponential flow} with possibly different constants.

\begin{lemma}
\label{lemma2}
For any initial condition, there exists a function $s$ to make the following equation hold
\begin{equation}
    \mbox{Re}(z_i)= \mbox{Re}(z\pm \sqrt{w})=e^{\frac{t}{2}} x_0 \pm s(t) e^\frac{t}{2} = e^\frac{t}{2}(x_0\pm s(t)),
    \label{Re(z_i)}
\end{equation}where $s(t)$ is independent of $t$ and positive for $t$ sufficiently large enough, and similarly, we have $$ \mbox{Im}(z_i)=\mbox{Im}(z\pm \sqrt{w})= e^\frac{t}{2}(y_0\pm \sigma(t)),$$ where $\sigma(t)$ is independent of $t$ for $t \gg 0$.
\end{lemma}
For example, for the initial conditions $z_1,z_2 \in i\mathbb{R}$, (Re$(z_1)$,Re$(z_2)$) flows to $(-\infty, +\infty)$.

\begin{remark}
    In fact, we expect that this construction also works for higher symmetric products Sym$^n$ since the smoothing of singularity can be given by a convolution with a smooth function.
\end{remark}


\begin{defn}
       We now define a function $\Delta: \mathbb{C} \to \mathbb{C}_{\text{Re}> 0}$ given by the pull back along the flow $\Delta(\sqrt{w(0)})= s(t)+i\sigma(t)$ for any $t \gg 0$. Moreover, we define the function $c = \text{Re}(\Delta)$.
\end{defn}
 
 Let $\epsilon$ be the radius of the support of the perturbation $\widetilde{\Phi}$ of the 
 plurisubharmonic function $\Phi$. 
\begin{lemma}
\label{funcc}
    We have the following properties of the function $c: \mathbb{C} \to \mathbb{R}_{>0}$
    \begin{enumerate}
        \item The function $c$ is smooth 
        \item The function $c$ is even 
        \item We have $| Re| \leq c \leq \max({|Re|, \epsilon})$ always hold. In particular, $c= |Re|$ holds if $| Re| >\epsilon>0$.
    \end{enumerate}
\end{lemma}

\begin{proof}
    First, it is obviously even since $\Delta$ is even as a function of $\sqrt{w(0)}$. \\
    Secondly,  if $|Re(\sqrt{w(0)})|=|Re(\frac{z_1-z_2}{2}(0))|> \epsilon$, the real part will always increasing along the flow, therefore, it will never go through the support of the perturbation. Since $Re (\sqrt{w(t)})=e^t Re(\sqrt{w_0})$, $c(\sqrt{w_0})=|Re(\sqrt{w_0})|$ by the definition of $c$. Then $c=|Re|$ holds for $|Re(\sqrt{w(0)})|> \epsilon$.\\
    Let's show the smoothness. 
  We have the equation 
    \begin{equation}
       c(\sqrt{w_0}))=e^{-\frac{t}{2}}c(\sqrt{w(t)})\label{pullback}
    \end{equation}
    For a fixed sufficiently large $t>0$, we will have $c(\sqrt{w_0}))=e^{-\frac{t}{2}} Re(\sqrt{w_t})$ for $\sqrt{w_0}$ in any compact subset of $\mathbb{C}$. Since the flow $\widetilde{\varphi}$ is smooth, $e^{-\frac{t}{2}} Re(\sqrt{w_t})$ is a smooth function of $\sqrt{w_0}$ and $c$ is also smooth on any given compact subset. Hence, $c$ is smooth everywhere.\\
    If $|Re(\sqrt{w_0})|< \epsilon$, by the lemma \ref{escaping lemma}, $Re(\sqrt{w_t})$ goes to infinity eventually.  There exists a time $t_0>0$ such that $Re(\sqrt{w_{t_0}})=\epsilon$.
    For $t>t_0$, the flow is unperturbed, so $Re(\sqrt{w(t)})=e^{\frac{t-t_0}{2}}Re(\sqrt{w_{t_0}})=e^\frac{t-t_0}{2}\epsilon$. Then, we have $c(\sqrt{w_{t_0}})=e^{-\frac{t}{2}}(e^{\frac{t-t_0}{2}}\epsilon)=e^{-\frac{-t_0}{2}}\epsilon \leq \epsilon$. Hence, we can conclude that $c \leq \max({|Re(\sqrt{w_0})|, \epsilon})$
\end{proof}

As we know for any $t$ large enough, $z_i(t)= x_i e^\frac{t}{2}+i y e^{-\frac{3t}{2}}$, where
 $x_i= Re(z_0) \pm c(\sqrt{w_0})$ by the equation \eqref{Re(z_i)} in lemma \ref{lemma2} and the definition of $c$.

\begin{defn}
    Let $U_{-,-}$ be the stable manifold of $(-\infty,-\infty)$, $U_{-,+}$ be the stable manifold of $(-\infty,+\infty)$, and $U_{+,+}$ be the stable manifold of $(+\infty,+\infty)$. In addition, we let $H_{0,-}$ be the stable manifold of $(-\infty, 0)$ and $H_{0,+}$ be the stable manifold of $(0, \infty)$. By stable manifold, we mean the set of points which converge to a given limit under the downward gradient flow.

\end{defn}

 The hypersurface $H_{0,\pm}$ corresponds to the case when one $x_i$ is $0$ and the sign of the other is positive or negative. Moreover. the sign of $x_i$ will determine which $U_{\pm,\pm}$ the point is flowing into eventually. 

\begin{proposition}
\label{hyperchar}
The characterization of these two hypersurfaces $H_{0,\pm}$ are $H_{0,-}=\{ Re(z_0) +c(\sqrt{w_0})=0\}$ and  $H_{0,+}=\{ Re(z_0) -c(\sqrt{w_0})=0\}$. 
For any point in $H_{0,-}$, we have $Re(z_0+\sqrt{w_0})\in [-\epsilon,0]$ and for any point in $H_{0,+}$,  we have
$Re(z_0- \sqrt{w_0})\in [0, \epsilon]$.
\end{proposition}
\begin{proof} The result then follows from part $3$ of Lemma \ref{funcc}. 
\end{proof}

\begin{remark}
    $H_{0,-} \cap H_{0,+} = \emptyset $ and this means that the piece $U_{-,+}$ has two disjoint boundaries.
\end{remark}

To show each piece of $U_{\pm,\pm}$ is indeed a Liouville sector, we start by constructing an $I$ function and showing that it satisfies the conditions of the definition of Liouville sector.

For a point $(z_1, z_2) \in Sym^2(\mathbb{C})$ near the boundaries  $H_{0,\pm}$, let us assume $|Re(z_1)| \leq |Re(z_2)|$ , we want to construct a tubular neighborhood $\mathcal{N}_{0,\pm}$ containing $H_{0,\pm}$.

We denote the downward gradient flow of the function $\widetilde{\Phi}$ by $\Psi_t$. 

\begin{defn}
\begin{enumerate} 

    \item Let $V_{-} := \{Re(z_1) \in (-\epsilon, \epsilon) \mbox{ and }Re(z_2) <- 2\epsilon \}$ and $V_{+} := \{Re(z_1) \in (-\epsilon, \epsilon) \mbox{ and }Re(z_2) > 2\epsilon \}$.
    \item Let $\mathcal{N}_{0,\pm} := \bigcup_{t\geq 0} \Psi^{-1}_{t}(V_\pm)$, these are all the points that will flow to $V_\pm$ after some time $t \geq 0$.
\end{enumerate} 
\end{defn}

For any point in $V_\pm$, it will not hit the perturbed area under the flow. Hence, we can just pull back along the flow to know the imaginary part. If $z_1(t_0),z_2(t_0))\in V_\pm$, then we have
\begin{equation}
  \label{eq3}
    Im(z_1(t))=e^{-\frac{3}{2}(t-t_0)} Im(z_1(t_0)). 
\end{equation}
for all $t>t_0$.

\begin{lemma}
$\mathcal{N}_{0,\pm}$ is an open neighborhood of $H_{0,\pm}$ . 
\end{lemma}
\begin{proof}
Since $V_\pm$ is open and $\Psi^{-1}_{t}(V_\pm)$ is open, $\mathcal{N}_{0,\pm}$ is also open. For a point $a$ on the hypersurface $H_{0,\pm}$, it will eventually flows to $V_\pm$ by lemma \ref{escaping lemma} for some $t_a>0$. Hence, by pulling back along the gradient flow, we have $a \in \Psi^{-1}_{t_a}(V_\pm) \subset \bigcup_{t\in \mathbb{R_+}} \Psi^{-1}_{t}(V_\pm) = \mathcal{N}_{0,\pm}$.
\end{proof} 

\begin{proposition}
    These two neighborhoods are disjoint, i.e. $\mathcal{N}_{0,-} \cap  \mathcal{N}_{0,+} =\emptyset$.
\end{proposition}




\begin{defn}
    We can define a function $I: \mathcal{N}_{0,-} \cup \mathcal{N}_{0,+}  \to \mathbb{R}$ given by $I=Im(z_1)$ for $(z_1,z_2) \in V_\pm$ and $I=e^{\frac{3}{2}t}Im(z_1(t))$ for $(z_1,z_2) \in \Psi^{-1}_{t}(V_\pm)$. This is independent of choices of $t$ by \eqref{eq3}.
\end{defn}

\begin{lemma}
    $ZI= \frac{3}{2} I$
    \label{lemmai1}
\end{lemma}
\begin{proof}
    Since $\Psi$ is the gradient flow of $-Z$, we only need to show that $\frac{d}{ds} {\Psi^*_s}(I)=-\frac{3}{2}I$.\\
    Let $t$ be large enough that $(z_1,z_2)\in \Psi^{-1}_{t}(V_\pm)$.
    By $\eqref{eq3}$,
    \begin{align*}
            \frac{d}{ds} {\Psi^*_s}(I)=\frac{d}{ds} I\circ \Psi_s &= \frac{d}{ds}(e^{\frac{3}{2}t}Im(z_1(s+t)) \\&=\frac{d}{ds}(e^{\frac{3}{2}t-\frac{3}{2}(s+t-t)}Im(z_1(t))))\\
    &=\frac{d}{ds}(e^{\frac{3}{2}(t-s)}Im(z_1(t))) \\
    &=-\frac{3}{2} e^{\frac{3}{2}(t-s)}Im(z_1(t))\\ 
    &=-\frac{3}{2}e^{\frac{3}{2}t}Im((z_1(s+t)) \\
    &= -\frac{3}{2}I
    \end{align*}
\end{proof}

\begin{lemma} $dI|_{\mbox {char. fol.}}>0$
 \label{lemmai2}
\end{lemma}

\begin{proof}
    The characteristic foliation inside $V_\pm$  is $\Ker{\omega|_{H_{0,\pm}\cap V_\pm}}=\Ker{\omega|_{\{Re(z_1)=0\}}}$ which is the imaginary $y_1$ direction. Since for $(z_1,z_2)\in H_{0,\pm}\cap V_\pm$, $I=Im(z_1)=y_1$, $dI$ is always positive on the characterisitc foliation of $H_{0,\pm}\cap V_\pm$. For the characteristic foliation in $\mathcal{N}_{0,\pm}$, since there exists some $t_0>0$ such that $\Psi_{t_0}(\mathcal{N}_{0,\pm}) \subset V_\pm $, $\Psi_{t_0}$ maps the characteristic foliation to itself and scales $dI$ by $e^{-\frac{3}{2}t_0}$. Hence, $dI$ will still be positive. 
\end{proof}

Since the hypersurfaces $H_{0,\pm}$ were defined as stable manifolds using the gradient flow of the Liouville vector field $Z$, we have:
\begin{lemma}
 \label{lemmai3}
    $Z$ is tangent to $H_{0,\pm}$.
\end{lemma}
\qed

\begin{theorem}
    $U_{-,-}$, $U_{-,+}$, and $U_{+,+}$ are Liouville sectors.  We have the decomposition $Sym^2(\mathbb{C})=U_{-,-} \cup_{H_{0,-}} U_{-,+} \cup_{H_{0,+}} U_{+,+}$.
\end{theorem}

\begin{proof}
It follows from lemma \ref{lemmai1}, lemma \ref{lemmai2}, and lemma \ref{lemmai3}.
\end{proof}

\subsection{Visualization of sectorial hypersurfaces}
Proposition \ref{hyperchar} gives us the characterization of hypersurfaces $H_{+}$ and $H_{-}$.

If we fix Im$(z_1)=$Im$(z_2)=a \in \R$, we can visualize the two hypersurfaces as in the first picture. If we fix $z_1= b \in \mathbb{R}$, we can draw the $z_2$ plane as follows. 
\vspace{-1cm}{
   \fontsize{9pt}{10pt}\selectfont
   \def\svgwidth{5.8in}
   \begin{center}
\begingroup%
  \makeatletter%
  \providecommand\color[2][]{%
    \errmessage{(Inkscape) Color is used for the text in Inkscape, but the package 'color.sty' is not loaded}%
    \renewcommand\color[2][]{}%
  }%
  \providecommand\transparent[1]{%
    \errmessage{(Inkscape) Transparency is used (non-zero) for the text in Inkscape, but the package 'transparent.sty' is not loaded}%
    \renewcommand\transparent[1]{}%
  }%
  \providecommand\rotatebox[2]{#2}%
  \newcommand*\fsize{\dimexpr\f@size pt\relax}%
  \newcommand*\lineheight[1]{\fontsize{\fsize}{#1\fsize}\selectfont}%
  \ifx\svgwidth\undefined%
    \setlength{\unitlength}{595.27559055bp}%
    \ifx\svgscale\undefined%
      \relax%
    \else%
      \setlength{\unitlength}{\unitlength * \real{\svgscale}}%
    \fi%
  \else%
    \setlength{\unitlength}{\svgwidth}%
  \fi%
  \global\let\svgwidth\undefined%
  \global\let\svgscale\undefined%
  \makeatother%
  \begin{picture}(1,1.41428571)%
    \lineheight{1}%
    \setlength\tabcolsep{0pt}%
    \put(0,0){\includegraphics[width=\unitlength,page=1]{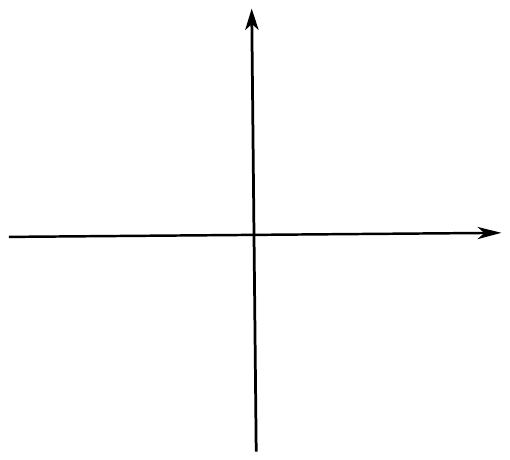}}%
    \put(0.46129919,1.29707202){\color[rgb]{0,0,0}\makebox(0,0)[lt]{\lineheight{1.25}\smash{\begin{tabular}[t]{l}Re($z_2)$\end{tabular}}}}%
    \put(0.64846293,1.12433857){\color[rgb]{0,0,0}\makebox(0,0)[lt]{\lineheight{1.25}\smash{\begin{tabular}[t]{l}Re($z_1$)\end{tabular}}}}%
    \put(0,0){\includegraphics[width=\unitlength,page=2]{drawing-3.pdf}}%
    \put(0.27525551,1.03060379){\color[rgb]{0,0,0}\makebox(0,0)[lt]{\lineheight{1.25}\smash{\begin{tabular}[t]{l}$ U_{-,-}$\end{tabular}}}}%
    \put(0.48008392,1.07993959){\color[rgb]{0,0,0}\makebox(0,0)[lt]{\lineheight{1.25}\smash{\begin{tabular}[t]{l}$ U_{-,+}$\end{tabular}}}}%
    \put(0.57886311,1.18668689){\color[rgb]{0,0,0}\makebox(0,0)[lt]{\lineheight{1.25}\smash{\begin{tabular}[t]{l}$ U_{+,+}$\end{tabular}}}}%
    \put(0,0){\includegraphics[width=\unitlength,page=3]{drawing-3.pdf}}%
    \put(0.07001351,1.23780623){\color[rgb]{0.20784314,0.12941176,0.86666667}\makebox(0,0)[lt]{\lineheight{1.25}\smash{\begin{tabular}[t]{l}Im$(z_1) = $Im$(z_2) = a$\end{tabular}}}}%
    \put(0.47248418,0.98064188){\color[rgb]{0.20784314,0.12941176,0.86666667}\makebox(0,0)[lt]{\lineheight{1.25}\smash{\begin{tabular}[t]{l}$H_{0,-}$\end{tabular}}}}%
    \put(0.4653073,1.24778726){\color[rgb]{0.20784314,0.12941176,0.86666667}\makebox(0,0)[lt]{\lineheight{1.25}\smash{\begin{tabular}[t]{l}$H_{0,+}$\end{tabular}}}}%
    \put(0.18641027,0.72188847){\color[rgb]{0,0,0}\makebox(0,0)[lt]{\lineheight{1.25}\smash{\begin{tabular}[t]{l}$ U_{-,-}$\end{tabular}}}}%
    \put(0.28457735,0.78518873){\color[rgb]{0,0,0}\makebox(0,0)[lt]{\lineheight{1.25}\smash{\begin{tabular}[t]{l}$ U_{-,+}$\end{tabular}}}}%
    \put(0.58254227,0.79046922){\color[rgb]{0,0,0}\makebox(0,0)[lt]{\lineheight{1.25}\smash{\begin{tabular}[t]{l}$ U_{-,+}$\end{tabular}}}}%
    \put(0.77908961,0.72621416){\color[rgb]{0,0,0}\makebox(0,0)[lt]{\lineheight{1.25}\smash{\begin{tabular}[t]{l}$ U_{-,+}$\end{tabular}}}}%
    \put(0.68645082,0.72179764){\color[rgb]{0,0,0}\makebox(0,0)[lt]{\lineheight{1.25}\smash{\begin{tabular}[t]{l}$ U_{+,+}$\end{tabular}}}}%
    \put(0.87757377,0.72792684){\color[rgb]{0,0,0}\makebox(0,0)[lt]{\lineheight{1.25}\smash{\begin{tabular}[t]{l}$ U_{+,+}$\end{tabular}}}}%
    \put(-0.00010448,0.73276766){\color[rgb]{0,0,0}\makebox(0,0)[lt]{\lineheight{1.25}\smash{\begin{tabular}[t]{l}$ U_{-,-}$\end{tabular}}}}%
    \put(0,0){\includegraphics[width=\unitlength,page=4]{drawing-3.pdf}}%
    \put(0.44982485,0.86155686){\color[rgb]{0,0,0}\makebox(0,0)[lt]{\lineheight{1.25}\smash{\begin{tabular}[t]{l}$U_{-,+}$\end{tabular}}}}%
    \put(0.44435558,0.66184325){\color[rgb]{0.12941176,0.11372549,0.94901961}\makebox(0,0)[lt]{\lineheight{1.25}\smash{\begin{tabular}[t]{l}$z_2$ plane\\for $b=0$\\\end{tabular}}}}%
    \put(0,0){\includegraphics[width=\unitlength,page=5]{drawing-3.pdf}}%
    \put(0.24622648,0.66022671){\color[rgb]{0.12941176,0.11372549,0.94901961}\makebox(0,0)[lt]{\lineheight{1.25}\smash{\begin{tabular}[t]{l}$z_2$ plane \\for $b \in [ -\epsilon,0)$\\\end{tabular}}}}%
    \put(0.19669412,0.81636414){\color[rgb]{0,0,1}\makebox(0,0)[lt]{\lineheight{1.25}\smash{\begin{tabular}[t]{l}$H_{0,-}$\end{tabular}}}}%
    \put(0,0){\includegraphics[width=\unitlength,page=6]{drawing-3.pdf}}%
    \put(0.00650495,0.81417795){\color[rgb]{0,0,1}\makebox(0,0)[lt]{\lineheight{1.25}\smash{\begin{tabular}[t]{l}$H_{0,-}$\end{tabular}}}}%
    \put(0,0){\includegraphics[width=\unitlength,page=7]{drawing-3.pdf}}%
    \put(0.82316055,0.66623031){\color[rgb]{0.12941176,0.11372549,0.94901961}\makebox(0,0)[lt]{\lineheight{1.25}\smash{\begin{tabular}[t]{l}$z_2$ plane \\for $ b \in (\epsilon,\infty)$\\\end{tabular}}}}%
    \put(0,0){\includegraphics[width=\unitlength,page=8]{drawing-3.pdf}}%
    \put(0.62708716,0.66513797){\color[rgb]{0.12941176,0.11372549,0.94901961}\makebox(0,0)[lt]{\lineheight{1.25}\smash{\begin{tabular}[t]{l}$z_2$ plane\\for $b \in (0,\epsilon]$\\\end{tabular}}}}%
    \put(0.67129284,0.80733141){\color[rgb]{0,0,1}\makebox(0,0)[lt]{\lineheight{1.25}\smash{\begin{tabular}[t]{l}$H_{0,+}$\end{tabular}}}}%
    \put(0.85851673,0.79410938){\color[rgb]{0,0,1}\makebox(0,0)[lt]{\lineheight{1.25}\smash{\begin{tabular}[t]{l}$H_{0,+}$\end{tabular}}}}%
    \put(0.09873244,0.78291068){\color[rgb]{0,0,0}\makebox(0,0)[lt]{\lineheight{1.25}\smash{\begin{tabular}[t]{l}$ U_{-,+}$\end{tabular}}}}%
    \put(0.37118923,0.79067401){\color[rgb]{0,0,1}\makebox(0,0)[lt]{\lineheight{1.25}\smash{\begin{tabular}[t]{l}$H_{0,-}$\end{tabular}}}}%
    \put(0.4858419,0.79067401){\color[rgb]{0,0,1}\makebox(0,0)[lt]{\lineheight{1.25}\smash{\begin{tabular}[t]{l}$H_{0,+}$\end{tabular}}}}%
    \put(0.00049479,0.65613733){\color[rgb]{0,0,1}\makebox(0,0)[lt]{\lineheight{1.25}\smash{\begin{tabular}[t]{l}$z_2$ plane\\for $ b \in (-\infty, -\epsilon)$\\\end{tabular}}}}%
  \end{picture}%
\endgroup%

   \end{center}
}
\vspace{-9cm}

\begin{corollary}
\label{sym2}
    We have $U_{-,-}\subset Sym^2({Re(z)\leq 0})$ and $U_{-,-}$ is a deformation retract of $Sym^2({Re(z)\leq 0})$.
\end{corollary}




\section{Existence of a quadratic Stein structure}

\begin{defn}
    Let $\Sigma $ be a Riemann surface and $\varphi$ be a proper plurisubharmonic function on $\Sigma$. Let $\{s_i\}_{i\in I}$ be the set of saddles of $\varphi$ and  $\mathcal{N}(\gamma_i)$ be the tubular neighborhood of the stable manifold $\gamma_i$ of the saddle $s_i$.\\
    We say $(\Sigma, \varphi)$ is a Riemann surface with a quadratic Stein structure if $\varphi |_{\mathcal{N}(\gamma_i)}$ is quadratic in local coordinates.
\end{defn}


\begin{proposition}
\label{global}
    For any  non compact topological surface $\Sigma$ with disjoint proper embedded arcs $\{\gamma_i \}, i \in I$,
    we can build a quadratic Stein structure with one saddle $s_i$ (and possibly other saddles as well)
    and on each $\gamma_i$ and one minimum $m_j$ 
    on each component of $\Sigma - \bigcup \gamma_i$. 
\end{proposition}

\begin{proof}
First, given $\Sigma$ and disjoint proper embedded arcs, we can add more nonseparating arcs to make the complement of arcs union of polygons. 

We let $D_1:=\{ z \in \mathbb{C}| Re(z) \leq \frac{4}{3}\}$ and let $D_n:=\{ z \in \mathbb{C}|z^n \in D_1\}$. 

We want to define a single smooth function $f:\mathbb{C} \to \mathbb{R}$ piecewisely. We let $f=f_1$ on $x<\frac{1}{3}$ and $f_2$ on $x>\frac{2}{3}$, where $f_2(x+yi)= \frac{1-\alpha}{2}(x-1)^2+\frac{\alpha}{2}y^2$ has a saddle at $1$, and $f_1$ goes to infinity if $Re(z)$ goes to negative infinity, has a minimum at $0$, $f_1(z=x+iy)=\frac{\alpha}{2}x^2+\frac{\alpha}{2}y^2-C $ for some constant $C$. 
To make sure that $f$ is plurisubharmonic, we also require $f_{xx}\geq 1-\alpha$ to hold everywhere. This is suitable for a choice of constant $C$.

We let $\varphi_1 =f|_{D_1}$. Let $\varphi_n(z):= \varphi_1(z^n)$ except in a neighborhood of $0$ and in the neighborhood of $0$, be a convex nondegenerate function of $|z|$. $\varphi_n$ defines a Stein structure on $D_n$ which is quadratic in the $n$ regions where $Re(z^n)>\frac{2}{3}$.

Since we added enough arcs to make the complement of arcs as a union of polygons, we can build $\Sigma$ from a union of domain $\bigcup D_{n_j}$ by gluing them to each other by identifying the quadratic regions via the biholomorphic map $z^{n_i} \mapsto 2-z^{n_k}$. And the functions $\varphi_{n_j}$ glue together to determine the Stein structure on $\Sigma$.
\end{proof}

\begin{defn} \label{D_1minus}
    $D_1^-:=\{ z \in \mathbb{C}| Re(z) \leq 1\}$ with a saddle sector structure given by restriction of the quadratic Stein structure of $D_1$.
\end{defn}

\begin{remark}
    If $\Sigma$ has non-empty boundary, the proposition also holds as long as there are no boundary component homeomorphic to $S^1$.
    \label{withbdry}
\end{remark}

\section{Construction of sectorial decomposition}
\subsection{Statement and notations}
\begin{theorem} \label{mainthm}
Let $(\Sigma,\varphi)$ be a Riemann surface with a quadratic Stein structure and let $\{s_i\}$ the set of saddles and $\{m_j\}$ the set of minimum. This determines a sectorial collection of hypersurfaces $H_{s_i}$ in $Sym^2(\Sigma)$, which decompose $Sym^2(\Sigma)$ into a union of sectors with corners $U_{m_i,m_j}$.
\end{theorem}

Let $\gamma_{s_i}$ be the stable manifold of the saddle $s_i$ and let $\Gamma^{2\epsilon}_{s_i}$ be the neighborhood band of $\gamma_{s_i}$, which is identified with $\{-2\epsilon <Re(z) <2\epsilon\}$. Let $D_{m_j}$ be the components of $\Sigma - \cup \gamma_{s_i}$.

We want to define $\widetilde{\Phi}$ on $Sym^2(\Sigma)$. 
Away from the diagonal, we let it be $\varphi(z_1)+\varphi(z_2)$ and smoothing of this near the diagonal. When both of $z_1, z_2 \in \Gamma^{2\epsilon}_{s_i}$, we can choose the smoothing to be identical to the local model in Section $2$.

The critical points of $\widetilde{\Phi}$ are $(m_i,m_j)$ (possibly $i=j$) with index $0$ and $(s_i, m_j)$ are the critical points of index $1$ and $(s_i, s_j)$ where $i\neq j$ are the critical points of index $2$. Let $U_{m_i,m_j}$ be the closure of the stable manifold of $(m_i,m_j)$, $H_{s_i, m_j}$ be the closure of the stable manifold of $(s_i, m_j)$, and $C_{s_i,s_j}$ be the stable manifold of $(s_i, s_j)$. For example, $U_{m_i,m_j}$ is the stable manifold of $(m_i,m_j)$ together with the strata $H_{s_k,m_j}$ and $H_{s_l,m_i}$ where $s_k$ is any saddle adjacent to $m_i$ and $s_l$ is any saddle adjacent to $m_j$.
 
It is clear to see that $C_{s_i,s_j}=H_{s_i} \cap H_{s_j}$ which decomposes $H_{s_i}$ into the union of $H_{s_i,m_j}$, where $H_{s_i}=\cup H_{s_i,m_j}$ where $m_j$ ranges over all minima. Since the Stein structure is a product away from the diagonal, we can see that $C_{s_i,s_j}=\gamma_{s_i} \times \gamma_{s_j}$.

   \begin{corollary}
    Let $m$ be the number of arcs and $n$ be the number of the connected components of $\Sigma - \cup_{i\in I} \gamma_i$.
       The decomposition constructed consists of $\binom{n+1}{2}$ Liouville sectors with corners separated by $mn$ pieces of smooth hypersurfaces that meet transversely at $\binom{m}{2}$ corners.
    \end{corollary}

\begin{example} \label{3saddle}
    Let's consider a Riemann surface with a quadratic Stein structure $(\mathbb{C}, \varphi)$ such that $\varphi$ has two saddles and three minima.
\end{example}
We can have the picture as follows to see the notations we just defined in the last subsection in this example. 
\vspace{-3cm}
 {
 \fontsize{10pt}{12pt}\selectfont
   \def\svgwidth{2.5in}
   \begin{center}
\begingroup%
  \makeatletter%
  \providecommand\color[2][]{%
    \errmessage{(Inkscape) Color is used for the text in Inkscape, but the package 'color.sty' is not loaded}%
    \renewcommand\color[2][]{}%
  }%
  \providecommand\transparent[1]{%
    \errmessage{(Inkscape) Transparency is used (non-zero) for the text in Inkscape, but the package 'transparent.sty' is not loaded}%
    \renewcommand\transparent[1]{}%
  }%
  \providecommand\rotatebox[2]{#2}%
  \newcommand*\fsize{\dimexpr\f@size pt\relax}%
  \newcommand*\lineheight[1]{\fontsize{\fsize}{#1\fsize}\selectfont}%
  \ifx\svgwidth\undefined%
    \setlength{\unitlength}{431.42168661bp}%
    \ifx\svgscale\undefined%
      \relax%
    \else%
      \setlength{\unitlength}{\unitlength * \real{\svgscale}}%
    \fi%
  \else%
    \setlength{\unitlength}{\svgwidth}%
  \fi%
  \global\let\svgwidth\undefined%
  \global\let\svgscale\undefined%
  \makeatother%
  \begin{picture}(1,0.99549152)%
    \lineheight{1}%
    \setlength\tabcolsep{0pt}%
    \put(0,0){\includegraphics[width=\unitlength,page=1]{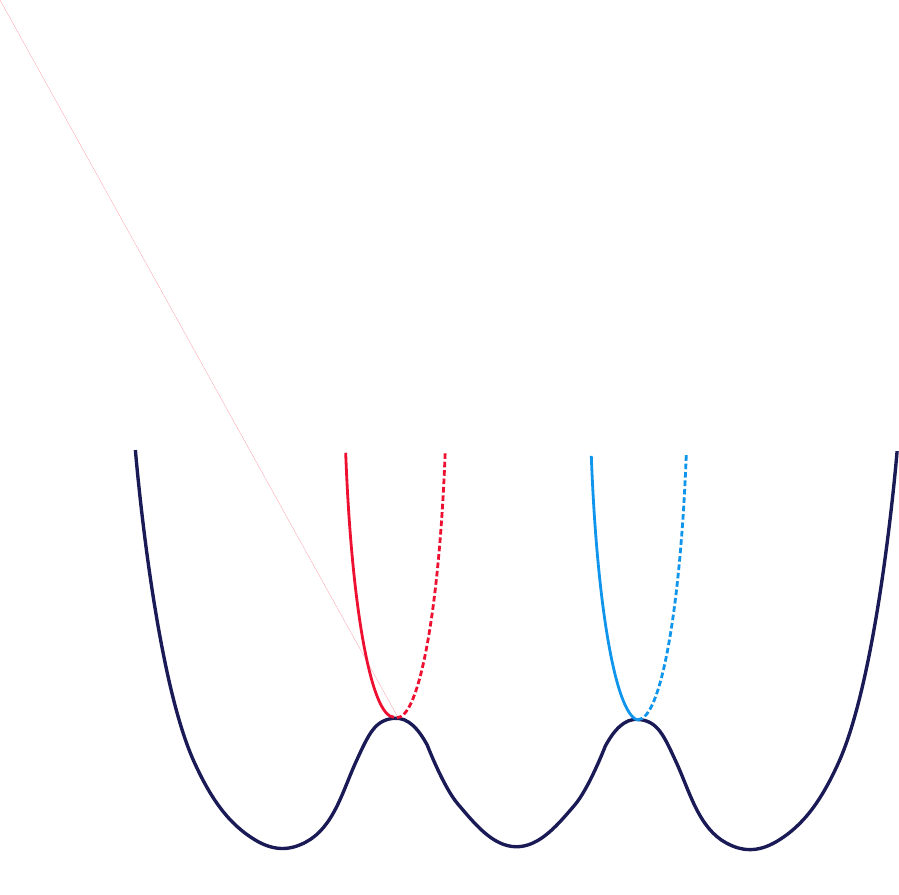}}%
    \put(0.2687232,0.00744912){\color[rgb]{0,0,0}\transparent{0}\makebox(0,0)[lt]{\lineheight{1.25}\smash{\begin{tabular}[t]{l}$m_1$\\\end{tabular}}}}%
    \put(0.52859395,0.00826061){\color[rgb]{0.50196078,0.50196078,0.50196078}\transparent{0}\makebox(0,0)[lt]{\lineheight{1.25}\smash{\begin{tabular}[t]{l}$m_2$\\\end{tabular}}}}%
    \put(0.79426981,0.00728059){\color[rgb]{0.50196078,0.50196078,0.50196078}\transparent{0}\makebox(0,0)[lt]{\lineheight{1.25}\smash{\begin{tabular}[t]{l}$m_3$\\\end{tabular}}}}%
    \put(0.42772212,0.1311735){\color[rgb]{0.50196078,0.50196078,0.50196078}\transparent{0}\makebox(0,0)[lt]{\lineheight{1.25}\smash{\begin{tabular}[t]{l}$s_1$\\\end{tabular}}}}%
    \put(0.68708692,0.13265921){\color[rgb]{0.50196078,0.50196078,0.50196078}\transparent{0}\makebox(0,0)[lt]{\lineheight{1.25}\smash{\begin{tabular}[t]{l}$s_2$\\\end{tabular}}}}%
  \end{picture}%
\endgroup%

   \end{center}
}
And we can visualize the decomposition of $Sym^2(\Sigma)$ as in the following picture.
 {
 \fontsize{10pt}{12pt}\selectfont
   \def\svgwidth{2.5in}
   \begin{center}
\begingroup%
  \makeatletter%
  \providecommand\color[2][]{%
    \errmessage{(Inkscape) Color is used for the text in Inkscape, but the package 'color.sty' is not loaded}%
    \renewcommand\color[2][]{}%
  }%
  \providecommand\transparent[1]{%
    \errmessage{(Inkscape) Transparency is used (non-zero) for the text in Inkscape, but the package 'transparent.sty' is not loaded}%
    \renewcommand\transparent[1]{}%
  }%
  \providecommand\rotatebox[2]{#2}%
  \newcommand*\fsize{\dimexpr\f@size pt\relax}%
  \newcommand*\lineheight[1]{\fontsize{\fsize}{#1\fsize}\selectfont}%
  \ifx\svgwidth\undefined%
    \setlength{\unitlength}{516.62018543bp}%
    \ifx\svgscale\undefined%
      \relax%
    \else%
      \setlength{\unitlength}{\unitlength * \real{\svgscale}}%
    \fi%
  \else%
    \setlength{\unitlength}{\svgwidth}%
  \fi%
  \global\let\svgwidth\undefined%
  \global\let\svgscale\undefined%
  \makeatother%
  \begin{picture}(1,0.67972719)%
    \lineheight{1}%
    \setlength\tabcolsep{0pt}%
    \put(0,0){\includegraphics[width=\unitlength,page=1]{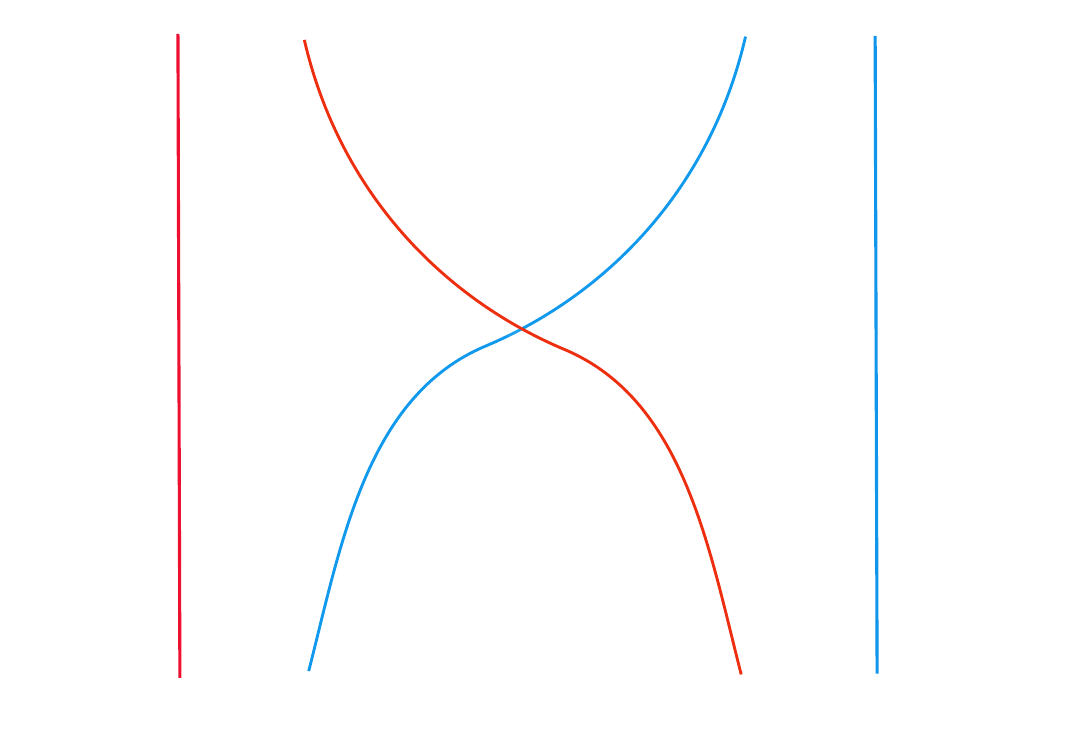}}%
    \put(0.24070853,0.39792503){\color[rgb]{0.12941176,0.0745098,0.0745098}\makebox(0,0)[lt]{\lineheight{1.25}\smash{\begin{tabular}[t]{l}$U_{m_1, m_2}$\end{tabular}}}}%
    \put(0.42548383,0.55149375){\color[rgb]{0.09019608,0.02352941,0.02352941}\makebox(0,0)[lt]{\lineheight{1.25}\smash{\begin{tabular}[t]{l}$U_{m_2, m_2}$\end{tabular}}}}%
    \put(0.44534472,0.18707254){\color[rgb]{0.11372549,0.01568627,0.01568627}\makebox(0,0)[lt]{\lineheight{1.25}\smash{\begin{tabular}[t]{l}$U_{m_1, m_3}$\end{tabular}}}}%
    \put(0.62002723,0.40735452){\color[rgb]{0.05490196,0.03137255,0.03137255}\makebox(0,0)[lt]{\lineheight{1.25}\smash{\begin{tabular}[t]{l}$U_{m_2, m_3}$\end{tabular}}}}%
    \put(-0.00076979,0.37628384){\color[rgb]{0.10196078,0.02745098,0.02745098}\makebox(0,0)[lt]{\lineheight{1.25}\smash{\begin{tabular}[t]{l}$U_{m_1, m_1}$\end{tabular}}}}%
    \put(0.84191445,0.36526164){\color[rgb]{0.02352941,0.02352941,0.02352941}\makebox(0,0)[lt]{\lineheight{1.25}\smash{\begin{tabular}[t]{l}$U_{m_3, m_3}$\end{tabular}}}}%
    \put(0.08573269,0.00388322){\color[rgb]{0.21568627,0.10980392,0.10980392}\makebox(0,0)[lt]{\lineheight{1.25}\smash{\begin{tabular}[t]{l}$H_{s_1, m_1}$\end{tabular}}}}%
    \put(0.28589696,0.03136061){\color[rgb]{0.09019608,0.02352941,0.02352941}\makebox(0,0)[lt]{\lineheight{1.25}\smash{\begin{tabular}[t]{l}$H_{s_2, m_1}$\end{tabular}}}}%
    \put(0.58931135,0.03136063){\color[rgb]{0.09411765,0.05098039,0.05098039}\makebox(0,0)[lt]{\lineheight{1.25}\smash{\begin{tabular}[t]{l}$H_{s_1, m_3}$\end{tabular}}}}%
    \put(0.7838454,0.02700548){\color[rgb]{0.14901961,0.0745098,0.0745098}\makebox(0,0)[lt]{\lineheight{1.25}\smash{\begin{tabular}[t]{l}$H_{s_2, m_3}$\end{tabular}}}}%
    \put(0.2336342,0.66286913){\color[rgb]{0.02352941,0.02352941,0.02352941}\makebox(0,0)[lt]{\lineheight{1.25}\smash{\begin{tabular}[t]{l}$H_{s_1, m_2}$\end{tabular}}}}%
    \put(0.64302592,0.66577256){\color[rgb]{0.02352941,0.02352941,0.02352941}\makebox(0,0)[lt]{\lineheight{1.25}\smash{\begin{tabular}[t]{l}$H_{s_2, m_2}$\end{tabular}}}}%
  \end{picture}%
\endgroup%

   \end{center}
}

\subsection{Proof of the theorem}
For each saddle $s_i$, we want to define $V_{s_i}:= \{ (z_1, z_2)\in Sym^2(\Sigma) | z_1 \in \Gamma^{\epsilon}_{s_i}, z_2 \notin \Gamma^{2\epsilon}_{s_i} \}$ where $\Gamma^{\epsilon}_{s_i}$ is a subset of $\Gamma^{2\epsilon}_{s_i}$ which is identified with $(-\epsilon, \epsilon) \times \mathbb{R}$. Let $\mathcal{N}_{s_i}=\cup_{t\geq 0} {\Psi}^{-t}(V_{s_i})$, where $\Psi_t$ is the downward gradient flow of the function $\widetilde{\Phi}$.
$\mathcal{N}_{s_i}$ is an open neighborhood of $H_{s_i}$ for each $s_i$.

\begin{defn}
    For each saddle $s_i$, we define a function $I_i: \mathcal{N}_{s_i} \to \mathbb{R}$ given by $I=Im(z_1)$ if $(z_1, z_2) \in V_{s_i}$ and $I=e^{\frac{3}{2}t} Im (z_1(t))$ for $(z_1,z_2) \in {\Psi}^{-t}(V_{s_i})$ for some $t>0$. This is independent of the choice of $t$.
\end{defn}

\begin{proposition}\label{Ifctproperties}
    The following properties hold. 
\begin{enumerate}
    \item $ZI_i= \frac{3}{2} I_i$
 \item $dI_i|_{\mbox {char. fol.}}>0$
    \item $Z$ is tangent to $H_s$.
\end{enumerate}
\end{proposition}
Because the Stein structure restricted to the band $\Gamma^{2\epsilon}_{s_i}$ can be identified with the local model in section $2$, these properties follow from our study of the local model.

Let us recall the definition of sectorial covering and sectorial hypersurfaces (Definition \ref{sector_covers} and definition 12.2 in \cite{MR4695507}).

{\emph{Claim.}} We want to show $\{H_{s_i}\}$ is indeed a sectorial covering by proving the three properties in the definition hold.  

\begin{proof}
\begin{enumerate}
    \item $dI_i|_{char.fol.(H_{s_i})}\ne 0$ is true by Proposition \ref{Ifctproperties}.
    \item  $dI_i|_{char.fol.(H_{s_j})}=0$\\
Let $\mathcal{N}_{ij}:=\cup_{t>0} \Psi^{-t}(V_{s_i} \cap V_{s_j})$ and this is a neighborhood of $C_{s_i,s_j}$, since we have $H_{s_i} \cap H_{s_j} = C_{s_i,s_j} =\gamma_{s_i} \times \gamma_{s_j} \subset \mathcal{N}_{ij} \subset \mathcal{N}_{s_i}  \cap \mathcal{N}_{s_j}$. There exists $t>0$ s.t. $I_i= e^{\frac{3t}{2}}\Psi^*_t(Im(z_1))$ and $I_j= e^{\frac{3t}{2}}\Psi^*_t(Im(z_2))$, where $z_1(t)\in \Gamma_{s_i}^\epsilon$ and $z_2(t)\in \Gamma_{s_j}^\epsilon$. Since the foliation inside $V_{s_j}$ is $Ker \omega|_{H_{s_j}\cap V_{s_j}}$ which is the imaginary $y_2$-direction. Since for $(z_1,z_2) \in H_{s_i} \cap V_{s_i}, I_i=Im(z_1)=y_1$, $dI_i$ is always zero on the characteristic foliation of $H_{s_j} \cap V_{s_j}$. For each point $s$ in $\mathcal{N}_{ij}\cap H_{s_j}$, there exists some $t_s>0$ such that $\Psi_{t_s}(s)\in V_{s_j} \cap V_{s_i}$, then $\Psi_{t_s}$ maps the characteristic foliation of $H_{s_j}$ to itself and scales $dI_i$ by $e^{-\frac{3t_s}{2}}$. Hence, $dI_i$ will still be zero on the characteristic foliation of $H_{s_j}$. 

\item $\{I_i,I_j\}= 0$. \\
Using the same notation as in the previous part, we have 
\begin{align*}
    \{I_i,I_j\}&= \{ e^{\frac{3t}{2}}\Psi^*_t(Im(z_1)),e^{\frac{3t}{2}}\Psi^*_t(Im(z_2))\}\\
    &= e^{\frac{3t}{2}}\Psi^*_t  \{Im(z_1),Im(z_2)\}
    \\ &=0.
\end{align*}
The last equality holds since $z_1(t)\in \Gamma_{s_i}^\epsilon$ and $z_2(t)\in \Gamma_{s_j}^\epsilon$.
    
\end{enumerate}
\end{proof}

\subsection{Symmetric product of 2-dimensional sectors}
  \begin{corollary}
        For any two dimensional Liouville sector $X$, we can smooth Sym$^2(X)$ to a Liouville sector (with corners if $\partial X$ has more than one component). 
    \end{corollary}

\begin{proposition}
     Every Liouville sector $X$ of complex dimension $1$ is deformation equivalent to a piece of a quadratic Stein Riemann surface.
\end{proposition}

\begin{lemma}
    Any two dimensional Liouville sectors structures on a given surface with boundary $(\Sigma, \partial)$ are deformation equivalent.
  \label{2dsector}  
\end{lemma}

\begin{proof}
    Let $X_0:= (\Sigma, \lambda_0, I_0)$ and $X_1:= (\Sigma, \lambda_1, I_1)$ be two Liouville sectors of dimension $2$. Let $X_0$ be with the symplectic form $\omega_0= d\lambda_0$ and the $I$-function $I_0$ on the boundary and $X_1$ be with the symplectic form $\omega_1= d\lambda_1$ and $I$-function $I_1$ on the boundary. 

    Let $\omega_t= d \lambda_t = (1-t) \omega_0 +t \omega_1$ and let $Z_t$ be the Liouville vector field of $\lambda_t=(1-t)\lambda_0 + t\lambda_1$, and $Z_t = \widetilde{\omega_t}^{-1}(\lambda_t)$ where $\widetilde{\omega}$ is the canonical isomorphism from the tangent bundle to the cotangent bundle given by $\omega$.
    
 We have $\widetilde{\omega_0}^{-1}(\lambda_t)=(1-t)\widetilde{\omega_0}^{-1}(\lambda_0)+t\widetilde{\omega_0}^{-1}(\lambda_1)$. We can see that $Z_t$ is pointing outward near infinity since $Z_0= \widetilde{\omega_0}^{-1}(\lambda_0)$ is pointing outward near infinity and $\widetilde{\omega_0}^{-1}(\lambda_1)$ is proportional to $\widetilde{\omega_1}^{-1}(\lambda_1)=Z_1$ which is also pointing outward near infinity. 

To show $X_0$ and $X_1$ are deformation equivalent, we need to show that the following three properties hold for a suitable function $I_t$ we construct below.
 \begin{enumerate}
     \item $Z_t$ is tangent to the boundary $\partial X_i$ outside of a compact subset.
     \item $dI_t|_{\partial X_i} >0$.
     \item $Z_t I_t = \alpha_t I_t$ outside of a compact set for some $\alpha_t>0$.
 \end{enumerate}

 First, we have $\widetilde{\omega_0}^{-1}(\lambda_t)=(1-t)\widetilde{\omega_0}^{-1}(\lambda_0)+ t\widetilde{\omega_0}^{-1}(\lambda_1)$. We know $Z_0= \widetilde{\omega_0}^{-1}(\lambda_0)$ is tangent to the boundary $\partial X_i$ at infinity, and $\widetilde{\omega_0}^{-1}(\lambda_1)$ is also tangent to the boundary $\partial$ at infinity since it is proportional to $Z_1$ which is tangent to the boundary at infinity.

Let $I_t:= (1-t)I_0+tI_1$ on a large compact set $S$ which contains $I_0^{-1}(0)$ and $I_1^{-1}(0)$ such that $Z_0$ and $Z_1$ both tangent to the boundary $\partial X_i$. We can then extend $I_t$ from $S$ to the whole boundary $\partial X_i$ by defining it using the Liouville vector field $Z_t$ such that $Z_t I_t =\alpha_t I_t$ where $\alpha_t= (1-t)\alpha_0 +t\alpha_1$. We only need to show that properties $2$ and $3$ hold for the compact set $S$ since it holds automatically via the flow of $Z_t$ outside of the set $S$. 

Secondly, $dI_t = (1-t) dI_0 + t dI_{1}|_{\partial} >0$ holds within the compact set $S$. On the complement of $S$, since $I_0^{-1}(0)\in S$ and $I_1^{-1}(0) \in S$, we have $I_0$ and $I_1$ both positive on $\mathbb{R}_+ - S$ and both negative on $\mathbb{R}_- - S$. And we know $dI_0>0$ and $dI_1>0$, we can get $dI_t>0$ outside of $S$ as well.
 
For the last property, $Z_t I_t = \alpha_t I_t$ outside of a compact set for some $\alpha_t>0$ is automatically true by construction.


\end{proof}

\begin{proof}[Proof of proposition]
This is true as the result follows from the lemma \ref{2dsector} and remark \ref{withbdry}. We use the construction of section $4$ to build a quadratic Stein sector on a surface diffeomorphic to $\Sigma$ and use the lemma to show the two sectors are deformation equivalent to each other.
\end{proof}

\begin{proof}[Proof of the Corollary]

    Step 1.  By the Proposition above, there exists a quadratic Stein structure on $X$ which has the desired local model 
   near each boundary arc. And we can arrange to have only one minimum $m$ in $X$ by Proposition \ref{global}.  We attach a half-plane $D_1$ along each boundary arc of $X$ and let $\Sigma := X \cup D_1$.\\
    Step 2.  By the main theorem $\ref{mainthm}$, we can get the decomposition $Sym^2(\Sigma) \cong \cup_{i,j} U_{m_i, m_j}$. And we can get $Sym^2(X)  \cong U_{m,m}$ as a deformation retract by the Corollary \ref{sym2}. 
    That is, we can smooth $Sym^2(X)$ to a Liouville sector $U_{m,m}$.
\end{proof}

\section{Fibers of stops and completions}
\subsection{Fibers of stops}
  Let $\Sigma$ be a Riemann surface with a quadratic Stein structure $\varphi$. And let $m$ be a minimum of $\varphi$. Let us cut the surface $\Sigma$ through the unstable manifolds $\gamma_i$ of each saddle, and we call the closure of the connected piece which contains the minimum $m$ to be $\Sigma_m$. 

\begin{proposition} \label{adjfiber}
    Let $s$ be a saddle and $m$ be a minimum adjacent to it. Then $H_{s,m} \subset \partial U_{m,m}$, and the fiber $F_{s,m}$ of the hypersurface $H_{s,m}$ is Liouville isomorphic to $\widehat{\Sigma_m}$ the convex completion at $\partial \Gamma_s^{2\epsilon}$ of $\Sigma_m - \Gamma_s^{2\epsilon}$.
\end{proposition}

\begin{proof}
The fiber $F_{s,m}: = I^{-1}_s(0) \cap H_{s,m}= \cup_{t>0} \Phi^{-t}(I_s^{-1}(0) \cap H_{s, m} \cap V_s)$.  Therefore, $I_s^{-1}(0) \cap H_{s, m} \cap V_s $ is a truncation of the fiber $F_{s,m}$. And we can also compute that
$$I_s^{-1}(0) \cap H_{s, m} \cap V_s =\{z_1=0, z_2 \notin \Gamma_s^{2\epsilon} | z_2 \in \Sigma_m\} \cong \{s\} \times (\Sigma_m - \Gamma_s^{2\epsilon})\subset Sym^2(\Sigma_m)$$ the two sides are Liouville isomorphic as sector domains, since the Stein structure here is just a Cartesian product.

 The fiber $F_{s,m}$ is a completion of  $I_s^{-1}(0) \cap H_{s, m} \cap V_s \cong \{s\}\times (\Sigma_{m}-\Gamma_s^{2\epsilon})$ whose completion is isomorphic to $\widehat{\Sigma_m}$.  
\end{proof}
\begin{proposition}\label{nadjfiber}
    Let $s$ be a saddle and $m_1$ be a minimum not adjacent to it. The fiber $F_{s,m_1}$ of $H_{s,m_1}$ is $\{s\} \times \Sigma_{m_1}$. 
\end{proposition}
\begin{proof}
Since $s$ and $m_1$ are not adjacent, the corresponding part of $H_{s,m_1}$ lies away from the diagonal.
Hence, locally near $s$, the hypersurface $H_{s,m_1}$ splits as a product

\[
I_s^{-1}(0)\cap H_{s,m_1}\cap V_s = \{z_1=0,\, z_2\in \Sigma_{m_1}\}
\cong \{s\}\times \Sigma_{m_1}.
\]
The Liouville structure is the Cartesian product one, so the same holds for the completed fiber:
$F_{s,m_1} = \bigcup_{t>0}\Phi^{-t}\!\left(I_s^{-1}(0)\cap H_{s,m_1}\cap V_s\right)
\cong \{s\}\times \Sigma_{m_1}.$
\end{proof}


\begin{example}
Let us recall the example with $3$ saddles in the example \ref{3saddle}. $s_1$ is adjacent to $m_1$ and $m_2$ and not adjacent to $m_3$. In this case, the fibers will be as follows by Proposition \ref{adjfiber} and Proposition \ref{nadjfiber}
\begin{enumerate}
    \item The fiber of $H_{s_1,m_1}$ in $\partial U_{m_1,m_1}$ is the Liouville completion of $\{s_1\}\times (\Sigma_{m_1}-\Gamma_{s_1}^{2\epsilon})$. 
    \item The fiber of $H_{s_1,m_2}$ in $\partial U_{m_2,m_2}$ is the Liouville completion of $\{s_1\}\times (\Sigma_{m_2}-\Gamma_{s_1}^{2\epsilon})$.
    \item The fiber of $H_{s_1,m_3}$ in $\partial U_{m_1,m_3}$ is the $\{s_1\} \times \Sigma_{m_3}$. 
\end{enumerate}
\end{example}

\begin{proposition}
    The fiber of $H_s$ is the Liouville completion of $\{s\} \times (\Sigma - \Gamma_s^{2\epsilon})$.
\end{proposition}
\begin{proof}
    Taking the union over all minima, the proposition follows from Proposition \ref{adjfiber} and Proposition \ref{nadjfiber}.
\end{proof}

\subsection{Completions} 
In this section, we study the convex completion of sectorial pieces of the decompositions of $Sym^2(\Sigma)$, i.e. $U_{m,m}$ and $U_{m,m'}$.

Let $\Sigma_m^{-}=\Sigma_m - \Gamma_s^{2\epsilon}$ and let $\widehat{\Sigma_{m}}$ be the Liouville completion of $\Sigma_m^-$ which is equivalent to the convex completion of $\Sigma_m$ in section 2.2.

\begin{proposition}\label{sym}
    For any minimum $m$, $\widehat{U_{m,m}}$ and $Sym^2(\widehat{\Sigma_{m}})$ are Liouville isomorphic.
\end{proposition}

\begin{proof}
Let $U_{m,m}^- $ be the subset of $U_{m,m}$ bounded by the hypersurface defined by Re$(z_1)=-\epsilon $ when Re$(z_2) \leq - 2\epsilon $,  Re$(z_2)=-\epsilon $  when Re$(z_1) \leq - 2\epsilon $, and Re$(z_1) +$ Re$(z_2)= -3 \epsilon$ when $ -\frac{15}{8} \epsilon \leq $Re$(z_1) \leq -\frac{9}{8}\epsilon$ and $ -\frac{15}{8} \epsilon \leq $Re$(z_2) \leq -\frac{9}{8}\epsilon$ and a smooth interpolation between these when Re$(z_1) \in [-2\epsilon,-\frac{15}{8}\epsilon] \cup [-\frac{9}{8}\epsilon,- \epsilon] $ and Re$(z_2) \in [-2\epsilon,-\frac{15}{8}] \cup [-\frac{9}{8},- \epsilon]$. 

And the boundary of $U_{m,m}^-$ is smooth and transverse to the Liouville vector field. Therefore, $U_{m,m}^- $ is a sector domain.

By construction, we have $Sym^2(\Sigma_m^-)\subset U_{m,m}^- \subset Sym^2(\hat{\Sigma})$ and also $Sym^2(\Sigma_m^-)\subset U_{m,m}^- \subset \widehat{U_{m,m}}$. In both cases, the negative Liouville flow maps all points into $U_{m,m}^-$. We have $ U_{m,m}^-$ is a common truncation of both $Sym^2(\hat{\Sigma})$ and $\widehat{U_{m,m}}$. Therefore,$Sym^2(\hat{\Sigma})$ and $\widehat{U_{m,m}}$ are Liouville isotopic.  
\end{proof}



When we consider two different minima, we have the following proposition which is analogous to the previous one. 

\begin{proposition} \label{product}
    For any two different minima $m$ and $m'$, $m \neq m'$, $\widehat{U_{m,m'}}$ and $\widehat{\Sigma_{m}} \times \widehat{\Sigma_{m'}} $ are Liouville isomorphic.
\end{proposition}

\begin{proof}
    The proof is analogue to the previous one by constructing a common truncation $\Sigma_m^- \times \Sigma_{m'}^-$
\end{proof}
\section{Homological mirror symmetry}
Recall that we say that a symplectic
manifold $(X,\omega)$ and a complex manifold $Y$ are mirror in the sense of
Kontsevich's homological mirror symmetry (HMS) if there is a derived
equivalence $F(X, \omega) \cong D^bCoh(Y)$ between the Fukaya category of $X$ and the
category of coherent sheaves of $Y$.

In this section, we study
homological mirror symmetry for the second symmetric product of a $4$-punctured sphere.

\subsection{Symmetric product of four punctured sphere}

Let $\Sigma:=\mathbb{P}^1 - 4 \mbox{ points}$ be a four-punctured sphere with a separating arc $\gamma$, which decomposes $\Sigma$ into two pieces $\Sigma_{-}$ and $\Sigma_{+}$. These two pieces are a pair of pants with a stop $(\Sigma_{-}, s_1)$ and a cylinder with a stop $(\Sigma_{+}, s_2)$, glued to each other along a boundary arc.
Let $\Pi_2: =Sym^2(\Sigma)$ denote the second symmetric product of $\Sigma$ which is known to be a two-dimensional pair of pants. We can equip $\Sigma$ with a quadratic Stein structure $\varphi$ such that it has a saddle on $\gamma$ by section $4$. 

 The completion of $\Sigma_{-}$ is isomorphic to a pair of pants $P = \mathbb{P}^1 - \mbox{{3 points}}$ and the completion of $\Sigma_{+}$ is isomorphic to $\mathbb{C}^*$.
The quadratic Stein structure on $Sym^2(\Sigma)$ determines a sectorial decomposition into three sectorial pieces $U_{-,-}$, $U_{-,+}$, and $U_{+,+}$ by the decomposition theorem in section $5$. 

By proposition \ref{sym} in section $6$, we have $\widehat{U_{-,-}} \cong Sym^2(\widehat{\Sigma_{-}}) \cong Sym^2(P) \cong  (\mathbb{C}^*)^2$ and $\widehat{U_{+,+}} \cong Sym^2(\widehat{\Sigma_{+}}) \cong Sym^2(\mathbb{{C}^*}) \cong  \mathbb{C}\times \mathbb{C}^*$. By proposition \ref{product}, $\widehat{U_{-,+}} \cong \widehat{\Sigma_{-}} \times \widehat{\Sigma_{+}} \cong P \times \mathbb{C}^*$.

The first step is to check that the stop on $U_{+,+}$ corresponds to the first coordinate $u_1$ in $\mathbb{C} \times \mathbb{C}^*$. 

Recall that by example $2.20$ in \cite{MR4106794}, a Landau–Ginzburg model naturally gives rise to a Liouville sector. More precisely, given a holomorphic function $W: X \to \mathbb{C}$, the behavior of $W$ towards infinity specifies a stop on the contact boundary $\partial_\infty X$. Concretely, for $c \gg 0$, the fiber $W^{-1}(c)$ has a skeleton which can be viewed as a singular Legendrian in $\partial_\infty X$, and this Legendrian serves as the stop. 
Thus, one obtains a Liouville sector $(X,\mathfrak{f})$, where 
$
\mathfrak{f} = skel( W^{-1}(c)).
$

By Corollary~$3.9$ in \cite{MR4695507}, a sectorial boundary can be replaced by a stop given by its fiber or its skeleton. 
More precisely, if 
$
X \hookrightarrow \bar{X} \hookleftarrow F \times \mathbb{C}_{\mathrm{Re}\geq 0}
$
is the inclusion of a Liouville sector into its convex completion with complement, 
then there is a quasi-equivalence
$
\mathcal{W}(X) \;\simeq\; \mathcal{W}(\bar{X},F_0)) \simeq \mathcal{W}((\bar{X}, \mathfrak{f})
$
where $F_0$ is the projection of the fiber to the infinity boundary and $\mathfrak{f}$ is the skeleton of $F_0$.

\begin{proposition}
     $U_{+,+} \cong (\mathbb{C} \times \mathbb{C}^*, W=u_1)$ as a Liouville sector.
\end{proposition}

    

\begin{proof}
Let $z$ be the coordinates on the completion $\widehat{\Sigma_+}\cong \C^*$, and for a pair of points $\{z_1,z_2\} \subset \widehat{\Sigma_+}$ we let
$
u_1:=z_1+z_2\in\C, u_2:=z_1z_2\in\C^* .
$
These are the symmetric coordinates on $Sym^2(\widehat{\Sigma_+})\cong \C\times \C^*$, hence
$\widehat{U_{+,+}}\cong \C\times \C^*$ by Proposition~\ref{sym}.

By construction of the sectorial decomposition (Section~$3$), the sectorial boundary of $U_{+,+}$ corresponds to the locus where
one point lies along the boundary arc of $\Sigma_+$ while the other point remains inside (up to a modification near the diagonal). Inside the sectorial boundary, the skeleton of the fiber $F =\{I=0\}$ corresponds to the locus where one point lies in the skeleton of $\Sigma_{+}$ and the other point is the saddle point of the boundary of $\Sigma_{+}$ which we take to be on the positive real axis.

By Corollary $3.9$ in \cite{MR4695507}, we can turn a Liouville sector into a Liouville domain with a mostly Legendrian stop. In this way, we obtain $Sym^2(\widehat{\Sigma_+})$ with 
the skeleton of the stop represented at infinity by
\[
\Lambda\;=\;\{\, (z_1, z_2)| z_2\to +\infty,\;\; z_1\in skel(\Sigma_{+})\}\subset\; \partial_\infty \widehat{U_{+,+}},
\]
which is a Legendrian stop for $\widehat{U_{+,+}}$.
The skeleton of $\Sigma_+$ is the set of points whose Liouville flow does not escape towards infinity or the boundary. This is a circle which we can take to be the unit circle.

Consider the holomorphic function
$
W:\C\times \C^*\longrightarrow \C, W(u_1,u_2)=u_1 .
$
For $c\gg 0$ the fiber
$
W^{-1}(c)=\{(z_1,z_2)\mid z_1+z_2=c\}
\subset \Sym^2(\widehat{\Sigma_+})
$
is a Stein submanifold whose skeleton corresponds to the descending manifolds of the critical points of the Stein potential function.

One can check that this skeleton corresponds to one point $z_1$ varying in a circle around the origin in $\mathbb{C}^*$ while the other point lies near $c$.

The stop induced by the Landau-Ginzburg model $(\C\times\C^*,W=u_1)$ agrees with the sectorial stop of $U_{+,+}$.
Geometrically, as $c\to+\infty$ the skeleton of $W^{-1}(c)$ converges to configurations where $z_1$ lies in the circle while $z_2$
runs out along the real positive direction, exactly the picture of the stop we obtained from the sectorial boundary of $U_{+,+}$.

Since $\widehat{U_{+,+}}\cong \C\times \C^*$ and the Legendrian stop on $\partial_\infty\widehat{U_{+,+}}$ coincides with
the stop $skel( W^{-1}(c))$ of the Landau-Ginzburg model $(\C\times\C^*,W=u_1)$. Ganatra--Pardon--Shende's
Corollary $3.9$ applies: a sectorial boundary can be replaced by the corresponding stop fiber without changing the wrapped Fukaya category.
Therefore, the two Liouville sectors
$
U_{+,+}$ and $ (\C\times \C^*,\, W=u_1)
$
agree up to Liouville deformation equivalence.

\end{proof}

\begin{corollary}
    $U_{-,+}\cup U_{+,+}$ is equivalent to the convex completion of $U_{-,+}$ along the sectorial boundary $H_{0,+}$ as a Liouville sector.
\end{corollary}

The key calculation is to show that $U_{-,-}$ is Liouville isotopic to $(\mathbb{C}^*)^2$ with a sectorial boundary at $z_1 + z_2 \to +\infty $ and $U_{-,+} \cup U_{+,+}$ is the product of a pair of pants and a cylinder with a stop.

\begin{proposition}
    $U_{-,-} \cong ((\mathbb{C}^*)^2, W=u_1+u_2)$ as a Liouville sector.
\end{proposition}

\begin{proof}
Let $z$ be the affine coordinate on the completion $\widehat{\Sigma_-}\cong P=\PP^1\setminus\{0,1,\infty\}$, and for a pair of points $\{z_1,z_2\}\subset \widehat{\Sigma_-}$ set
$u_1:=z_1z_2\in\C^*$ and $u_2:=-(z_1-1)(z_2-1)\in\C^*.$
These are symmetric coordinates on $\Sym^2(\widehat{\Sigma_-})\cong(\C^*)^2$, hence
$\widehat{U_{-,-}}\cong (\C^*)^2$ by Proposition~\ref{sym}. In these coordinates one has
$z_1+z_2 -1  \;=\; u_1+u_2$,
so the potential $W=u_1+u_2$ on $(\C^*)^2$ is the same as $W(u_1,u_2)=z_1+z_2-1$.

By construction of the sectorial decomposition (Section~3), the sectorial boundary of $U_{-,-}$ corresponds to the locus where
one point lies along the distinguished boundary end of $\Sigma_-$ while the other point remains inside (up to a modification near the diagonal). Inside the sectorial boundary, the skeleton of the boundary fiber $F=\{I=0\}$ corresponds to the locus where one point lies in the skeleton of $\Sigma_-$ and the other point is the boundary saddle which we take to be a point on the positive real axis close to $\infty$.

By Corollary~3.9 in Ganatra--Pardon--Shende, we can turn a Liouville sector into a Liouville domain with a mostly Legendrian stop. In this way, we obtain $\Sym^2(\widehat{\Sigma_-})$ with 
the skeleton of the stop represented at infinity by
\[
\Lambda\;=\;\{\, (z_1, z_2)\mid z_2\to +\infty,\;\; z_1\in \skel(\Sigma_{-})\,\}\subset\partial_\infty \widehat{U_{-,-}},
\]
which is a Legendrian stop for $\widehat{U_{-,-}}$.
Here $\skel(\Sigma_-)$ is the standard trivalent theta graph $\Gamma_P\subset P$.

Consider the holomorphic function
$
W:(\C^*)^2\longrightarrow \C$ given by $W(u_1,u_2)=u_1+u_2$. 
For $c\gg 0$, the fiber
$
W^{-1}(c)\;=\;\{(u_1,u_2)\mid u_1+u_2=c\}\ \subset\ \Sym^2(\widehat{\Sigma_-})
$
is a Stein submanifold whose skeleton is given by the descending manifolds of a quadratic Stein potential.

One can check that this skeleton corresponds to one point $z_1$ constrained to the pair-of-pants skeleton $\Gamma_P\subset P$ while the other point $z_2$ lies near $c$ and escapes along the chosen positive end.

The stop induced by the Landau--Ginzburg model $((\C^*)^2,W=u_1+u_2)$ agrees with the sectorial stop of $U_{-,-}$.
Geometrically, as $c\to+\infty$ the skeleton of $W^{-1}(c)$ converges to $z_1\in\Gamma_P$ while $z_2$
runs out along the positive real direction, exactly the picture of the stop we obtained from the sectorial boundary of $U_{-,-}$.

Since $\widehat{U_{-,-}}\cong (\C^*)^2$ and the Legendrian stop on $\partial_\infty\widehat{U_{-,-}}$ coincides with
the stop $\skel(W^{-1}(c))$ of the Landau-Ginzburg model $((\C^*)^2,W=u_1+u_2)$, Ganatra--Pardon--Shende's
Corollary~3.9 applies: a sectorial boundary can be replaced by the corresponding stop fiber without changing the wrapped Fukaya category.
Therefore, the two Liouville sectors
$U_{-,-}$ and$ ((\C^*)^2,\, W=u_1+u_2)$
agree up to Liouville deformation equivalence. \qedhere
\end{proof}

\begin{proposition}
    $ U_{-,+} \cup U_{+,+} \cong \widehat{\Sigma_{-}} \times \Sigma_{+}$ as a Liouville sector, i.e. a product of a pair of pants with a cylinder with one stop.
\end{proposition}

\begin{proof}
    By proposition \ref{product}, $\widehat{U_{-,+}} \cong \widehat{\Sigma_{-}} \times \widehat{\Sigma_{+}} \cong P \times \mathbb{C}^*$. By corollary 7.2, we attach $U_{+,+}$ to remove a stop corresponding to the sectorial boundary $H_{0,+}$. So we only need to find the stop that corresponds to the sectorial boundary $H_{0,-}$. As in the previous proof, it is enough to find the mostly Legendrian skeleton of the fiber $F_{0,-}$.
\end{proof}
\subsection{HMS of two-dimensional pair of pants}

As the wrapped Fukaya category of $U_{+,+}$ is zero, we can glue back $U_{+,+}$ to $U_{-,+}$ to remove a stop. We have 
$ U_{-,+} \cup U_{+,+} \cong \Sigma_{+} \times \Sigma_{-}$.

By the push-out diagram in Ganatra-Pardon-Shende \cite{MR4695507} to glue Liouville sectors, we have the following push-out diagram.

\[\begin{tikzcd}
	{\mathcal{W}(P)} & {\mathcal{W}(U_{-,+}\cup U_{+,+})} \\
	{\mathcal{W}(U_{-,-})} & {\mathcal{W}(\text{Sym}^2(\Sigma))}
	\arrow[from=1-1, to=1-2]
	\arrow[from=1-1, to=2-1]
	\arrow[from=1-2, to=2-2]
	\arrow[from=2-1, to=2-2]
\end{tikzcd}\]

By a classical result of algebraic geometry, we also have the following push-out diagram.
\[\begin{tikzcd}
	{D^b\text{Coh}(\{xy=0\})} & {D^b\text{Coh}(\{xy=0\}\times \mathbb{C})} \\
	{D^b\text{Coh}(\mathbb{C}_{z=0}^2)} & {D^b\text{Coh}(\{xyz=0\})}
	\arrow[from=1-1, to=1-2]
	\arrow[from=1-1, to=2-1]
	\arrow[from=1-2, to=2-2]
	\arrow[from=2-1, to=2-2]
\end{tikzcd}\]

Several of the categories appearing in these two diagrams are related to each other by homological mirror symmetry results for the pair of pants and for toric varieties. We have
 $\mathcal{W}(P) \cong {D^b\text{Coh}(\{xy=0\})}$ by Abouzaid-Auroux-Efimov-Katzarkov-Orlov \cite{MR3073884} and Gammage-Shende \cite{MR4578536}. 
We have ${\mathcal{W}(U_{-,-})} \cong {D^b\text{Coh}(\{\mathbb{C}^2\})}$ by Abouzaid and Fang-Liu-Treumann-Zaslow. 
In addition, we have
$\mathcal{W}(U_{-,+} \cup U_{+,+}) \cong \mathcal{W}(\widehat{\Sigma_{-}} \times \Sigma_{+}) \cong {D^b\text{Coh}(\{xy=0\}\times \mathbb{C})}$. by Gammage-Shende or by the mirror of products.
 
\begin{proposition} \label{comm1}
The following diagram commutes. 
    \[\begin{tikzcd}
	{\mathcal{W}(P)} & {\mathcal{W}(U_{-,-})} \\
	{D^b\text{Coh}(\{xy=0\})} & {D^b\text{Coh}(\mathbb{C}^2)}
	\arrow[from=1-1, to=1-2]
	\arrow[from=1-1, to=2-1]
	\arrow[from=1-2, to=2-2]
	\arrow[from=2-1, to=2-2]
\end{tikzcd}\]
\end{proposition}
\begin{proof}
    This is by Hanlon-Hicks \cite{MR4394684} and Gammage-Shende \cite{MR4578536}.
    This is by Hanlon-Hicks \cite{MR4394684} and Gammage-Shende \cite{MR4578536}.
\end{proof}
\begin{proposition} \label{comm2}
The following diagram commutes. 
\[\begin{tikzcd}
	{\mathcal{W}(P)} & {\mathcal{W}(U_{-,+}\cup U_{+,+})} \\
	{D^b\text{Coh}(\{xy=0\})} & {D^b\text{Coh}(\{xy=0\}\times \mathbb{C})}
	\arrow[from=1-1, to=1-2]
	\arrow[from=1-1, to=2-1]
	\arrow[from=1-2, to=2-2]
	\arrow[from=2-1, to=2-2]
\end{tikzcd}\]

\end{proposition}

\begin{proof}
    This is by Hanlon-Hicks and Gammage-Shende \cite{MR4578536}.
\end{proof}

\begin{theorem}
The homological mirror symmetry conjecture holds for the second symmetric product of the $4$ punctured sphere $\Sym^2(\mathbb{P}^1 - 4 \mbox{pts})$ and $\{xyz=0\}\subset \mathbb{C}^3$. More precisely, the wrapped Fukaya category of the second symmetric product of the $4$ punctured sphere is equivalent to the derived category of coherent sheaves on $\{xyz=0\}\subset \mathbb{C}^3$.
\end{theorem}

\begin{proof}
The vertical faces of the commutative diagrams follow from Proposition \ref{comm1} and Proposition \ref{comm2}.  We can use the following commutative diagram to show the theorem.

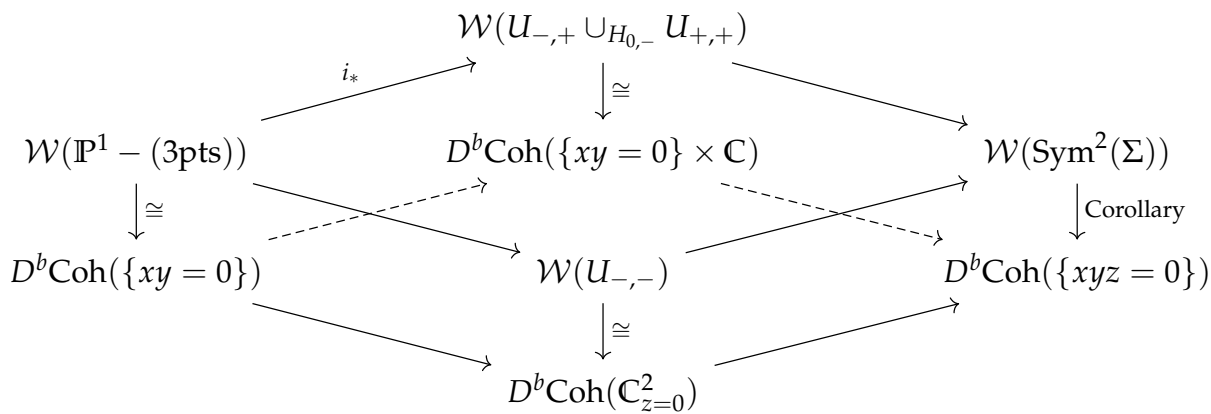
\begin{figure}[h]

 \[\begin{tikzcd}[scale=0.9]
&& {\mathcal{W}(U_{-,+} \cup_{H_{0,-}} U_{+,+})} \\
{\mathcal{W}(\mathbb{P}^1 - (3\text{pts}))} && {D^b\mathrm{Coh}(\{xy=0\}\times \mathbb{C})} && {\mathcal{W}(\Sym^2(\Sigma))} \\
{D^b \mathrm{Coh}(\{xy=0\})} && {\mathcal{W}(U_{-,-})} && {D^b \mathrm{Coh}(\{xyz=0\})} \\
&& {D^b \mathrm{Coh}(\mathbb{C}^2_{z=0})}
\arrow["\cong", from=1-3, to=2-3]
\arrow[from=1-3, to=2-5]
\arrow["{i_*}", from=2-1, to=1-3]
\arrow["\cong", from=2-1, to=3-1]
\arrow[from=2-1, to=3-3]
\arrow[dashed, from=2-3, to=3-5]
\arrow["{ \text{Corollary}}", from=2-5, to=3-5]
\arrow[dashed, from=3-1, to=2-3]
\arrow[from=3-1, to=4-3]
\arrow[from=3-3, to=2-5]
\arrow["\cong", from=3-3, to=4-3]
\arrow[from=4-3, to=3-5]
\end{tikzcd}\]
  \label{fig:mypic}
  \caption{HMS for the $2$-dimensional pair of pants}
  \label{marker}
\end{figure}
The categories at the right are pushouts, which are determined by the rest of diagram and therefore equivalent. 
\end{proof}

\bibliographystyle{alpha}
\bibliography{references}

\end{document}